\documentclass[a4paper,11pt]{amsart}
\usepackage{amssymb}

\usepackage[ansinew]{inputenc}    
\usepackage{graphicx}        
\usepackage{multicol}        
\usepackage[bottom]{footmisc}
\usepackage{extarrows}
\usepackage{float}

\usepackage{enumitem}
\usepackage{mathtools}
\usepackage{nicefrac}

\usepackage{tikz-cd}

\usepackage{color}
\definecolor{colordelink}{rgb}{0,0,0.50}
\definecolor{colordecite}{rgb}{0,0.5,0}
\definecolor{colordeurl}{rgb}{0,0.41,0.5}
\usepackage[pagebackref=true]{hyperref}
\hypersetup{
    colorlinks=true,
    linkcolor=colordelink,
    citecolor=colordecite,
		urlcolor=colordeurl,
}

\usepackage{ytableau}
\usepackage[capitalize,nameinlink,noabbrev]{cleveref}
\crefalias{assumption}{enumi}
\crefname{assumption}{Assumption}{Assumptions}
\Crefname{assumption}{Assumption}{Assumptions}
\def\rhd{\operatorname{rhd}}
\def\C{\mathbb C}
\def\Z{\mathbb Z}
\def\rk{\operatorname{rk}}
\def\tr{\operatorname{tr}}

\def\rond{\mathaccent"7017} 

\newtheorem{theorem}{Theorem}[section]
\newtheorem{corollary}[theorem]{Corollary}
\newtheorem{lemma}[theorem]{Lemma}
\newtheorem{proposition}[theorem]{Proposition}

\theoremstyle{definition}
\newtheorem{definition}[theorem]{Definition}

\theoremstyle{definition}
\newtheorem{remark}[theorem]{Remark}

\theoremstyle{definition}
\newtheorem{example}[theorem]{Example}

\begin{document}

\title{Thom condition and Monodromy}

\author{R. Giménez Conejero, Lê D\~ung Tráng and
J.J.~Nu\~no-Ballesteros}

\address{
Alfr\'ed R\'enyi Institute of Mathematics, Re\'altanoda utca 13-15,
H-1053 Budapest, 
Hungary
}
\email{Roberto.Gimenez@uv.es}

\address{Professor emeritus, University of Aix-Marseille, France}
\email{ledt@ictp.it}

\address{Departament de Matemàtiques,
Universitat de Val\`encia, Campus de Burjassot, 46100 Burjassot
Spain  and Departamento de Matem\'atica, Universidade Federal da Para\'\i ba, 
CEP 58051--900, Jo\~ao Pessoa - PB, Brazil}
\email{Juan.Nuno@uv.es}

\thanks{The first named author has been partially supported by MCIU Grant FPU16/03844. The third 
named author has been partially supported by MICINN Grant PGC2018--094889--B--I00 and by GVA Grant AICO/2019/024}


\begin{abstract} We give the definition of the Thom condition and we show that given any germ of 
complex analytic function $f\colon(X,x)\to(\C,0)$ on a complex 
analytic space $X$, there exists a geometric local monodromy without fixed points, provided that 
$f\in\mathfrak m_{X,x}^2$, where $\mathfrak m_{X,x}$ is the maximal ideal of $\mathcal O_{X,x}$. 
This result generalizes a well-known theorem of the second named author when $X$ is smooth and
proves a statement by Tibar in his PhD thesis. It also implies the A'Campo theorem that the Lefschetz number of the monodromy is equal to zero. 
Moreover, we give an application to the case that $X$ has maximal rectified homotopical depth at $x$ 
and show that a family of such functions with isolated critical points and constant total Milnor number 
has no coalescing of singularities.
\end{abstract}

\maketitle


\section*{Introduction}
\label{sec:0}

In \cite{Milnor1968} J. Milnor proved that for any germ of complex function:
$$f:({\mathbb C}^{n+1},x)\to ({\mathbb C},0)$$
one can associate a smooth locally trivial fibration for $1\gg \varepsilon >0$:
$$\varphi_{\varepsilon}:{\mathbb S}_\varepsilon(x)\setminus f^{-1}(0)\rightarrow {\mathbb S}^1$$
induced by $f/|f|$, where ${\mathbb S}_\varepsilon(x)$ is the sphere centered at
$x$ with radius $\varepsilon$ and ${\mathbb S}^1$ is the circle of radius
$1$ of ${\mathbb C}$ centered at the origin.

In \cite{Hamm1971} H. Hamm made the observation that, when $x$ is an isolated critical point of $f$, 
the fibration of Milnor is isomorphic to the local fibration, for $1\gg \varepsilon\gg\eta>0$:
$$\psi_{\varepsilon, \eta}:\rond{B}_\varepsilon(x)\cap f^{-1}({\mathbb S}_\eta)\rightarrow {\mathbb S}_\eta$$
induced by $f$, where $\rond{B}_\varepsilon(x)$ is the open ball centered at the point $x$ with radius $\varepsilon$.

From the work of \cite{Hamm1973a} (Th\'eor\`eme 1.2.1 p. 322) the hypothesis of isolated singularity can be lifted. Moreover the proper map:
$$\overline{\psi}_{\varepsilon, \eta}:B_\varepsilon(x)\cap f^{-1}({\mathbb S}_\eta)\rightarrow {\mathbb S}_\eta$$
is a locally trivial fibration.

Milnor's fibration leads to a notion of monodromy associated to $f$ at~$x$. Precisely let $\varphi:X\rightarrow{\mathbb S}^1$ be a proper locally trivial
smooth fibration. One can build on $X$ a smooth vector field $v$ which lifts the unit vector field tangent to ${\mathbb S}^1$.
The integration of this vector field defines a smooth morphism $h:F\to F$ of a fiber  of $\varphi$ onto itself that we call {\it a geometric 
 monodromy of~$\varphi$}. A geometric monodromy is not uniquely defined, but one can prove that its isotopy class is unique.
Therefore there is an isomorphism induced by a geometric monodromy of~$\varphi$ on the homology (or cohomology) of the fiber $F$
called {\it the monodromy of $\varphi$}.

In the case of Milnor's fibration one often use the terminology of {\it local geometric monodromy}
and {\it local monodromy} of $f$ at the point $x$.

In \cite{Trang1975} the second named author gave a proof of the fact that for any germ of complex analytic function:
$$f:({\mathbb C}^{n+1},x)\to ({\mathbb C},0)$$ 
having a 
critical point at $x$, there is a local geometric monodromy of $f$ at $x$ without fixed points.

By a well-known theorem of S. Lefschetz (see {\it e.g.} \cite[p. 179]{Hatcher2002}) this result implies 
that the local monodromy of $f$ at $x$ has a Lefschetz number
equal to $0$. In fact, in \cite{ACampo1973}, A'Campo showed that the Lefschetz number is zero in a more general 
situation: let $(X,x)$ be any germ of complex analytic space and denote by ${\mathfrak m}_{X,x}$ the the maximal ideal of 
the local ring $\mathcal O_{X,x}$ of germs of analytic functions of $X$ at $x$.


\begin{theorem}[\textit{cf.}  \cite{ACampo1973}]\label{A'Campo}
Let $f\colon(X,x)\to ({\mathbb C},0)$ be a germ of complex analytic function such that $f\in {\mathfrak m}^2_{X,x}$. 
Then the local monodromy of $f$ at $x$ has Lefschetz number equal to $0$.
\end{theorem}

A'Campo used heavy mathematical machinery to prove this result in \cite{ACampo1973} and attributed its proof to P. Deligne.

In this work we give the following generalization of Lê's theorem, which in particular implies \cref{A'Campo}:

\begin{theorem}\label{nofixedpoints}
Let $f\colon(X,x)\to ({\mathbb C},0)$ be a germ of complex analytic function such that $f\in {\mathfrak m}^2_{X,x}$. 
Then there is a local geometric monodromy of $f$ at $x$ which does not fix any point.
\end{theorem}

A big part of the argument in \cite{Trang1975} relies strongly on the fact that for a sufficiently generic linear form $\ell:{\mathbb C}^{n+1}\to {\mathbb C}$, the map~$\Phi=(\ell,f):(X,x)\to({\mathbb C}^{2},0)$ satisfies the Thom condition (see definition below) with respect to some convenient stratification. This allows to lift and integrate the plane vector field given by the carrousel in order to construct a local geometric monodromy in which we can apply an induction argument. It is well known that any complex analytic function $f\colon(X,x)\to ({\mathbb C},0)$ satisfies the Thom condition with respect to some stratification (see for instance \cite{Briancon1994,Hironaka1977}). However, this is not true in general when we consider maps~$(X,x)\to(\C^p,0)$, with $p>1$ (see \cref{ex:Sabbah}). 

Unfortunately, sometimes the Thom condition used to be ignored and some authors use it but without an explicit mention. So, we feel that it is important to emphasize this aspect of the theory. In \cref{sec2}, we show that given a map $\Phi=(g,f):(X,x)\to({\mathbb C}^{2},0)$ and a Whitney stratification of $X$, then $\Phi$ satisfies the Thom condition provided that:
\begin{enumerate}[label=\arabic*., ref=\arabic*]
	\item $g$ is the restriction of a submersion $\tilde{g}:U\to\mathbb{C}$;
	\item $f^{-1}(0)$ is union of strata;
	\item $\Gamma\coloneqq \overline{C(\Phi)\setminus f^{-1}(0)}$ is empty or a curve (i.e., it has dimension one), for $C(\Phi)$ the critical locus of $\Phi$;
	\item $\Phi^{-1}(0)\cap \Gamma\subseteq \left\{x\right\}$;
	\item for each stratum $S\in\mathcal{S}$ such that $\dim S \geq1$, $\ker D_x\tilde{g}\notin \nu_{\bar{S},f}^{-1}(x)$, where $\nu_{\bar{S},f}:C(\bar{S},f)\to\bar{S}$ is the relative conormal bundle of $\bar{S}$;
\end{enumerate}
see \cref{ThomGeneral}. In particular, all these conditions are easily satisfied when we consider $g$ as the restriction of a sufficiently generic linear form $\ell:{\mathbb C}^{n+1}\to {\mathbb C}$. 

Another important contribution of \cite{Trang1975} is the notion of privileged polydisks. This gives a fundamental system of neighbourhoods which is more convenient than the Euclidean balls if we want to proceed by induction on the dimension of $X$. We show in Section 4 how to adapt this notion in the case that $X$ is not smooth.

We remark that  a statement of \cref{nofixedpoints} already appeared in \cite{Tibar1993} (see also \cite{Tibar1992}), following the ideas of \cite{Trang1975} about relative polar curves and the carrousel construction.  However, the technical details about the Thom condition, the lifting and integration of the vector field or the construction of the privileged polydisks are not mentioned in \cite{Tibar1993}.  Here, we offer a complete and detailed explanation of all the steps in the proof.


As in \cite{Trang1975}, the proof of \cref{nofixedpoints} uses the notion of relative polar curve, which is 
due essentially to R. Thom. When $X=\C^{n+1}$ we first choose
 a sufficiently small open neighbourhood $U$ of $x$.  For almost all linear function $\ell:{\mathbb C}^{n+1}\to {\mathbb C}$, 
 one has that the critical space of the restriction $\left.(\ell,f)\right|_{U\setminus\{f=0\}}$ is either always empty or 
a non-singular curve. When it is non-empty, we call the closure of the critical space of $\left.(\ell,f)\right|_{U\setminus\{f=0\}}$ 
the relative polar curve 
$\Gamma_\ell(f,x)$ of $f$ at $x$ with respect to $\ell$. 

The remarkable property of the relative polar curve is that, when $f$ has a critical point at $x$, 
its image by $\left.(\ell,f)\right|_{U}$ is empty or a curve that 
Thom called the Cerf's diagram, which has as tangent cone the axis of values of $\ell$
(see {\it e.g.} \cite[ Proposition 6.7.5]{LNS}). 
We show in \cref{polar} how to adapt this construction to the case that $X$ is singular at $x$ by taking a Whitney 
stratification. The condition that $f\in {\mathfrak m}^2_{X,x}$ is used here in order to prove
that the tangent cone of the Cerf’s diagram is the $\ell$-axis.

Associated with the Cerf's diagram we have the \textit{carrousel}, a construction which  again appears in \cite{Trang1975}. 
This is a vector field $\omega$ over a small enough solid torus $D\times \partial D_\eta$ centered at the origin 
in $\C\times\C$ such that:
\begin{enumerate}[label=(\roman*)]
	\item its projection onto the second component gives a tangent vector field over $\partial D_\eta$ of
	length $\eta$ and positive direction (called in \cite{Trang1975} the \textit{unitary vector field} of $\partial D_\eta$),
	\item its restriction to $\left\{0\right\}\times\partial D_\eta$ is indeed the unitary vector field,
	\item for every component of the Cerf's diagram with reduced equation $\delta_\alpha=0$, $\omega$ is 
	tangent to every $\delta_\alpha=\epsilon$ with $\epsilon$ small enough, and
	\item the only integral curve that is closed after a loop in $\partial D_\eta$ is $\left\{0\right\}\times\partial D_\eta$.
\end{enumerate}

Now we can use techniques of stratification theory to lift the carrousel $\omega$ and obtain a stratified vector field on $X$ 
which is globally integrable. The integral curves of this vector field define a local geometric monodromy of $f$ at $x$ 
and of its restriction to $X\cap\{\ell=0\}$, defined on section 2. By condition (iv), the fixed points of the monodromy of $f$ 
can appear only on $X\cap\{\ell=0\}$. Thus, the proof of \cref{nofixedpoints} follows by induction on the dimension of $X$ at $x$.

We give an example that the condition that $f\in {\mathfrak m}^2_{X,x}$ is necessary, even if $f$ has critical 
point at $x$ in the stratified sense. In the last section, we also extend a well-known theorem of the second 
named author (see \cite{Trang1973}) about no coalescing of families of functions with isolated critical points 
and constant total Milnor number. The extension works when we consider functions on spaces with maximal rectified homotopical depth (also called spaces with 
Milnor property in \cite{Hamm-Le-Handbook}).


\section{Relative polar curves}\label{polar}

Let $f:(X,x)\to ({\mathbb C},0)$ be the germ of a complex analytic function. We still call $f:X\to {\mathbb C}$ a representative of this germ.
Let ${\mathcal S}=(S_\alpha)_{\alpha\in A}$ be a Whitney stratification of a sufficiently small representative ${X}$ of $(X,x)$.
 By Whitney stratification we mean a regular complex analytic stratification defined by H.Whitney in  
\cite[\S 19, p. 540]{Whitney1965}. In particular the strata $S_\alpha$ and their closures $\overline{S}_\alpha$ are
complex analytic spaces.
We can assume that $x$ is in the closure $\overline{S}_\alpha$ of all the strata $S_\alpha$. So, the set of indices $A$ is finite.

Using \cite[Lemma 21 \S3]{Trang2017} one can prove that there is a non-empty open 
Zariski subset $\Omega_{\alpha}$ of the space of affine functions such that, for every $\ell$ in $\Omega_{\alpha}$, $\ell(x)=0$  and the critical locus
$C_{\alpha}$ of $\left.(\ell, f)\right|_{S_{\alpha}\setminus f^{-1}(0)}$, the function induced by $(\ell, f)$ on the space $S_{\alpha} \setminus f^{-1}(0)$,
is either always empty or a non-singular curve. Then, the closure $\Gamma_{\alpha}$ of $C_{\alpha}$ in ${X}$ is either empty or a reduced curve. Furthermore, we can also show that one can choose the $\Omega_\alpha$'s such that the restriction 
$\left.(\ell, f)\right|_{\Gamma_{\alpha}}$ is finite for any $\alpha\in A$. We define (see for instance \cite[p. 310]{Trang1976});

\begin{definition}
For $\ell\in \bigcap_{\alpha\in A}\Omega_\alpha$ the union $\bigcup_{\alpha\in A}\Gamma_\alpha$ is either empty or a reduced curve. This curve
is called the {\it relative polar curve $\Gamma_\ell(f,{\mathcal S},x)$} of $f$ at $x$ relatively to $\ell$ and the stratification ${\mathcal S}$ of $X$.
\end{definition}

\begin{remark}
Notice that if the stratum $S_\alpha$ has dimension one, the whole stratum $S_\alpha$ is critical and $\Gamma_\alpha$ is the
closure $\overline{S}_\alpha$. In this case, since $S_\alpha$ is connected, $\Gamma_\alpha$ is a branch 
of the curve $\Gamma_\ell(f,{\mathcal S},x)$ at $x$, {\it i.e.} an analytically
irreducible curve at $x$.
\end{remark}


A theorem of Remmert implies that the image of $\Gamma_\alpha$ by $\left.(\ell, f)\right.$ is 
 either empty or a curve $\Delta_\alpha$, for any $\alpha\in A$ (see, for example, \cite[p. 5]{Bell1998}).

We define:

\begin{definition}
The union $\bigcup_{\alpha\in A}\Delta_\alpha$ is either empty or a reduced curve. When it is a curve, it
is called the {\it Cerf's diagram $\Delta_\ell(f,{\mathcal S}, x)$ of $f$ at $x$ relatively to $\ell$} and the stratification ${\mathcal S}$.
Otherwise we say that the Cerf's diagram of $f$ at $x$ relatively to $\ell$ and the stratification ${\mathcal S}$ is empty.
\end{definition}

When the stratification ${\mathcal S}$ is fixed, we shall speak of the relative polar curve $\Gamma_\ell(f,x)$ and the Cerf's diagram 
$\Delta_\ell(f,x)$ without mentioning the stratification ${\mathcal S}$. But the reader must be aware that the notion of polar curve and Cerf's diagram
depends on the choice of the stratification.
\vskip.1in
\vskip.1in
We shall go back and forth between the case $(\mathbb{C}^{n+1},x)$ and the general case of germs of reduced analytic spaces $(X,x)$ 
and compare them to generalize what we have in \cite{Trang1975}. For example, if $(X,x)=({\mathbb C}^{n+1},x)$, we can consider a 
Whitney stratification which has only one stratum. In \cite{Trang1975}, we have seen that the emptiness of $\Gamma_\ell(f,x)$ means 
that the Milnor fiber of $f$ at $x$ is diffeomorphic to the product of the Milnor fiber of $\left.f\right|_{\{\ell=0\}}$ at $x$ with an open disc, 
hence the local geometric monodromy of $f$ at $x$ is induced by the product of the local geometric monodromy of $\left.f\right|_{\{\ell=0\}}$ 
at $x$ and the identity of the open disc.

Also, for a germ of complex analytic function $f:(X,x)\to ({\mathbb C},0)$, in general, we may suppose that the hyperplane $\{\ell=0\}$ 
is transverse to all the strata of the Whitney stratification ${\mathcal S}$ and it induces a Whitney stratification on ${X}\cap \{\ell=0\}$. 
Then, using the same arguments of \cite{Trang2017}, we can prove the following:

\begin{proposition}
If, for a general linear form $\ell$ at $x$, the relative polar curve $\Gamma_\ell(f,x)$ is empty, there is a stratified homeomorphism of 
the Milnor fiber of $f$ at $x$ and the product with an open disc with the Milnor fiber of the restriction $\left.f\right|_{{X}\cap\{\ell=0\}}$ at $x$.
\label{emptycerf}
\end{proposition}

The proof of this proposition is based on the techniques Mather used to prove the Thom-Mather 
first isotopy lemma, {\it cf.} \cite{Mather2012,Gibson1976}. We will outline these techniques in the next section and use them later.

Now observe that when $(X,x)=({\mathbb C}^{n+1},x)$ the point $x$ is a critical point of $f$ if and only if 
$f\in{\mathfrak m}_{{\mathbb C}^{n+1},x}^2$, where ${\mathfrak m}_{{\mathbb C}^{n+1},x}$ is the maximal ideal of the local ring ${\mathcal O}_{{\mathbb C}^{n+1},x}$. 
In the case of a germ of complex analytic function on $(X,x)$, the hypothesis $f\in {\mathfrak m}_{X,x}^2$, where ${\mathfrak m}_{X,x}$
is the maximal ideal of ${\mathcal O}_{X,x}$, replaces the condition that $f$ is critical at $x$. In fact, a key result 
for the proof of \cref{nofixedpoints} is:

\begin{proposition}\label{prop tangent}
For a sufficiently general linear form $\ell$, if $f\in {\mathfrak m}_{X,x}^2$, every branch of the Cerf's diagram 
$\Delta_\ell (f,x)$ is tangent at the point $(0,0)$ to the first axis, the image by $(\ell,f)$ of $\{f=0\}$.
\end{proposition}

\begin{proof} Of course, we have a Whitney stratification ${\mathcal S}=(S_\alpha)_{\alpha\in A}$
on a sufficiently small representative of the germ $(X,x)$. We may assume that $x$ is in the closure of all the strata.

It is enough to prove the proposition for the image $\Delta_\alpha$ of $\Gamma_\alpha$ by $(\ell,f)$, for each $\alpha\in A$.

In \cite{Trang1975}, it was considered that $\ell$ is a coordinate of $\mathbb{C}^{n+1}$ to compare easily the growth of $f$ and $\ell$ 
along a component of the Cerf's diagram. We are going to give a similar proof for the $(\mathbb{C}^{n+1},x)$ case for any 
general linear form, and generalize it twice to reach our current context.

Suppose that $(X,x)=(\mathbb{C}^{n+1},x)$, for our purpose $\ell$ can be expressed as: 
$$\ell(v)=\left\langle v,a\right\rangle=\sum_{i=1}^{n+1} v_i\overline{a_i},$$ 
and we can assume that $\left\|a\right\|=1$. Let us define $H$ as the kernel of $\ell$ and, then, any vector of ${\mathbb C}^{n+1}$ can be written 
as a sum of a vector of $H$ and a multiple of the vector $a$ (note that $a$ is the unitary normal of $H$).

Now we can take a parametrization $p(t)$ of
 a branch of
 $\Gamma_\alpha$ and compare the growths of $f$ and $\ell$ there. Using de l'H\^opital's rule 
and identifying $\ell$ with its differential we have:
$$ \lim_{t\to 0} \frac{\left(f\circ p\right)(t)}{\left(\ell\circ p\right)(t)}=\lim_{t\to 0} 
\frac{\left(f\circ p\right)'(t)}{\left(\ell\circ p\right)'(t)}=\lim_{t\to 0} \frac{df_{p(t)}\big(p'(t)\big)}{\ell \big(p'(t)\big)}.$$

Now we can decompose $p'(t)$ as the sum of a vector of $H$, say $p_H'(t)$, and $\lambda a$. Hence:
$$\lim_{t\to 0} \frac{df_{p(t)}\big(p'(t)\big)}{\ell \big(p'(t)\big)} =\lim_{t\to 0} 
\frac{df_{p(t)}\big(p_H'(t)\big)+\lambda df_{p(t)}\left(a\right)}{\ell\big(p_H'(t)\big)+\lambda \ell\left(a\right)}.$$
Furthermore, we know that $df$ and $\ell$ are colinear along $p(t)$, and we have assumed $\ell(a)=\|a\|=1$, therefore

\begin{align*}
\lim_{t\to 0} \frac{df_{p(t)}\big(p_H'(t)\big)+\lambda df_{p(t)}\left(a\right)}{\ell\big(p_H'(t)\big)+\lambda \ell\left(a\right)}&=
\lim_{t\to 0}\frac{df_{p(t)}\left(a\right)}{\ell\left(a\right)} \\
&= df_x(a).
\end{align*}

At this point we see where the condition of $f\in {\mathfrak m}_{X,x}^2$ appears, because this last term is zero. This proves the tangency of
the statement in this context.

If we want the same result on $(X,x)\subset(\mathbb{C}^{N},x)$, with $(X,x)$ regular at $x$, the main problem is that 
$\ell$ is defined in ${X}$, and we cannot work with such a vector $a$ and space $H$. What we can do 
is to extend $\ell$ to the ambient space and work on the tangent bundle 
of ${X}$, hence we can choose a linear function $L:\mathbb{C}^{N}\to\mathbb{C}$ such that $\left.L\right|_{X}=\ell$ and $H'$ as the kernel 
of $L$. By genericity $T_x {X}$ is not contained in $H'$ so $H'\cap T_x {X}$ is a hyperplane of $T_x {X}$, say $H$. This happens, nearby $x$, 
for every tangent space along a parametrization of $\Gamma_\alpha$ (in this case there is only one strata), so we can reproduce 
the computations we did before.

Finally, if $(X,x)$ is general, we can still extend $\ell$ but we cannot work with the tangent bundle of ${X}$ any more 
({\it e.g.}, if $(X,x)$ is a Whitney umbrella at $x$ even $x$ is a strata by itself). To avoid this complication, firstly we shall find a convenient 
hyperplane of $\mathbb{C}^N$ for the role of $H'$ and then work with the extension of $f$ when needed. From now on, we 
will work with a strata $S_\alpha$, or its adherence, but for the sake of the similarity with the previous cases we will call
${Y}$ the closure $\overline{S_\alpha}$.

Therefore, our first step is to find an hyperplane to work with. For this purpose consider the (projective) conormal space $C({Y})$
of ${Y}$ in $\mathbb{C}^N$, this is given by the closure in $Y\times \check{\mathbb{P}}^{N-1}$ of the space:
$$ \left\{(q,H')\,|\, (q,H')\in {Y}^{reg} \times \check{\mathbb{P}}^{N-1}: T_q {Y}^{reg} \subset H'\right\}, $$
together with the conormal map $\nu: C({Y})\to {Y}$. 
It is a classical fact ({\it cf.} \cite[II.4.1]{Teissier1982}) that ${\dim}\,\nu^{-1}(x)\leq N-2$ or, being more specific, 
there is a hyperplane $H'$ outside $\nu^{-1}(x)$ and by continuity outside every fiber of $\nu$ over a neighbourhood of $x$ in $Y$.

Therefore, consider a linear form $L:\mathbb{C}^N\to\mathbb{C}$ with such a hyperplane $H'$ as kernel and define 
$\ell$ to be $\left.L\right|_{Y}$. Furthermore, since $f\in \mathfrak{m}_{Y,x}^2$, we can take an extension $F:(\mathbb{C}^N,x)\to(\mathbb{C},0)$ 
of $f$, such that $F\in \mathfrak{m}_{\mathbb{C}^N,x}^2$.

Finally, as we have done before, consider a parametrization $p(t)$ of a branch of $\Gamma_\alpha$ and compare the growths of $f\big(p(t)\big)$ and $\ell\big(p(t)\big)$. 
To do so define $a_t$ to be the unitary normal of the hyperplane $H_t \coloneqq H'\cap T_{p(t)} {Y}$ in $T_{p(t)} {Y}$, well defined by 
the previous election of $H'$, and $a_0$ its limit. Now we can keep proceeding as before and finish the computation with $F$, {\it i.e.}:

\begin{align*}
	\lim_{t\to0}\frac{\left(f\circ p\right)(t)}{\left(\ell\circ p\right)(t)}&=\lim_{t\to0} df_{p(t)}(a_t)\\
	&=\lim_{t\to0}dF_{p(t)}(a_t)\\
	&=dF_x(a_0)\\
	&=0.
\end{align*}

Note that the election of $H'$, for $S_\alpha$, was made in an open set. Since we have only a finite number of
strata to which $x$ is adherent, we can take a common $H'$ 
for every $S_\alpha$ and repeat the computation. This finishes the proof.
\end{proof}

%
%
%
%


\section{Thom condition for maps onto the plane}\label{sec2}

As it is well known, the Thom condition appears as hypothesis in many important results because it gives control on a map defined over stratified spaces. For example, it appears in the Thom-Mather isotopy lemmas. We recall now its definition (see also \cref{fig:t-thomcondition}).

Recall the definitions of stratified map and Thom map:

\begin{definition} \label{def:stratmap}
A map $f:X\to X'$ between Whitney stratified sets $X$ and $X'$ is a \emph{stratified map} if the restriction of the map on each stratum of $X$ is submersive onto a stratum of $X'$, \textit{i.e.}, $f(S_\alpha)\subseteq S'_\beta$ and $\left.f\right|_{S_\alpha}:S_\alpha\rightarrow S'_\beta$ is submersive where $S_\alpha$ is a stratum of $X$ and $S'_\beta$ is a stratum of $X'$. 
\end{definition}


\begin{definition}[{{\it cf.} \cite[II.2.5]{Gibson1976}}]\label{Thomconditiondef}
Let $f:N\to N'$ be a smooth map between manifolds $N$ and $N'$ and $X\subset N$ and $X'\subset N'$ two stratified sets such that $f(X)\subset X'$ and the induced map $\left.f\right|:X\to X'$ is a stratified map. We say that $S_t$ is \textit{Thom regular over $S_r$ at $p$ 
relatively to $\left.f\right|:X\rightarrow X'$} if any sequence of points  $\left\{q_n\right\}_n\subset S_t$ converging to $p\in S_r$ is such that
\begin{equation}
\ker D_p\left.f\right|_{S_r} \subseteq \lim_{n}\ker D_{q_n}\left.f\right|_{S_t},
\label{eq:Thomregcondition}
\end{equation} 
when the limit exists. If $\left.f\right|$ is Thom regular for any pair of strata we simply say that the pair of stratifications is a \textit{Thom stratification} of $\left.f\right|$ and that $\left.f\right|$ is a \textit{Thom map} or that it satisfies the \textit{Thom $a_{\left.f\right|}$ condition}. 
\end{definition}

\begin{figure}
	\centering
		\includegraphics[width=0.75\textwidth]{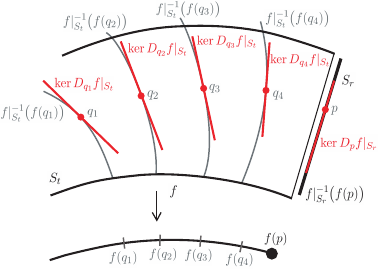}
	\caption{Representation of the Thom condition.}
	\label{fig:t-thomcondition}
\end{figure}

Not every map has a stratification such that it is a Thom map, for example the map $f:\mathbb{R}^2\rightarrow\mathbb{R}^2 $ so that $f(x,y)=(x,xy)$ does not admit a Thom stratification (see \cite[page 24]{Gibson1976}). However, if we consider complex functions $f:(X,x)\to (\mathbb{C},0)$, with $(X,x)$ a germ of a complex space, we have a result of existence of stratifications providing the Thom regularity condition given by Hironaka in \cite[Corollary 1]{Hironaka1977}. Furthermore, Brian\c{c}on, Maisonobe and Merle, in \cite[Theorem 4.2.1]{Briancon1994}, gave a result that assures the Thom condition provided the stratification is Whitney regular.

Let $(X,x)\subseteq(\mathbb{C}^N,x)$ be a germ of a complex analytic space. This contrast suggests that the case of a map germ $\Phi=(g,f):(X,x)\to (\mathbb{C}^2,0)$ is harder to study, as there could be maps that do not admit any Thom stratification (see also \cref{ex:Sabbah} below). In order to study this case we need another definition.

\begin{definition}
Given a complex space $S\subseteq\mathbb{C}^N$ and a function $f:S\to\mathbb{C}$, the \textit{conormal bundle of $\bar{S}$ relative to $f$} or, simply, the \textit{relative conormal bundle of $\bar{S}$} is the closed space $$ C\big(\bar{S},f\big)\coloneqq\overline{\left\{(p,H)\in S^{reg}\times \check{\mathbb{P}}^{N-1} | \ker D_qf \subseteq H \right\}} $$ together with the projection $\nu_{\bar{S},f}: C(\bar{S},f)\to \bar{S}$.
\end{definition}

\begin{remark}
Observe that $C(\bar{S},f)$ coincides with $C(\bar{S})$ given in the proof of \cref{prop tangent} when $f$ is constant.
\end{remark}

Furthermore, we will need to refine some stratifications to provide the Thom condition:

\begin{lemma}\label{chungo} Let $\varphi:V\to W$ be a smooth map, where $V\subset\mathbb R^n$ and $W\subset \mathbb R^m$ are open subsets. Assume that $\mathcal S$ is a Whitney stratification of a subset $X\subset V$ such that for all $S\in\mathcal S$, $\varphi|_S:S\to W$ is a submersion and that $\mathcal T$ is a Whitney stratification of $W$. Then,
\[
\mathcal S'=\big\{S\cap \varphi^{-1}(T):\ S\in\mathcal S,\ T\in\mathcal T\big\}
\]
is also a Whitney stratification of $X$.
\end{lemma}

\begin{proof}
Take a pair of strata $A=S\cap\varphi^{-1}(T)$ and $B=S'\cap\varphi^{-1}(T')$ such that $A\subseteq \bar{B}$, with $S,S'\in\mathcal S$ and $T,T'\in\mathcal T$. 
We factorize $\varphi$ as the composition:
\[
\begin{tikzcd}
V\arrow[r,"i"]&G(\varphi)\arrow[r,"\pi_2"]&W
\end{tikzcd}
\]
where $G(\varphi)\subset V\times W$ is the graph of $\varphi$, $i$ is the diffeomorphism given by $i(v)=\big(v,\varphi(v)\big)$ and $\pi_2(v,w)=w$. It follows that $B$ is Whitney regular over $A$ in $V$ if and only if $i(B)$ is Whitney regular over $i(A)$ in $G(\varphi)$, or equivalently, in $V\times W$. To prove this observe that we can write $i(A)$ and $i(B)$ in the form
\begin{equation}\label{inter}
i(A)=i(S)\cap(S\times T),\quad i(B)=i(S')\cap(S'\times T').
\end{equation}
Moreover, we know that $i(S')$ is Whitney regular over $i(S)$ and $S'\times T'$ is Whitney regular over $S\times T$.

Let $\{x_n\}$ and $\{y_n\}$ be sequences in $i(A)$ and $i(B)$ respectively, both converging to $x\in i(A)$. We also assume that $\overline{x_ny_n}$ converges to a line $L$ and that $T_{y_n}i(B)$ converges to a plane $E$ in the corresponding Grassmannians of $\mathbb R^n\times\mathbb R^m$. We have to show that $L\subset E$.

By taking subsequences if necessary, we can assume that $T_{y_n} i(S')$ converges to a plane $E_1$ and that 
$T_{y_n}(S'\times T')$ converges to another plane $E_2\times E_3$, again in the corresponding Grassmannians of $\mathbb R^n\times\mathbb R^m$. Since $i(S')$ is Whitney regular over $i(S)$ and $S'\times T'$ is Whitney regular over $S\times T$, we have $L\subset E_1\cap (E_2\times E_3)$.

From \eqref{inter} it follows that $E\subset E_1\cap (E_2\times E_3)$. Furthermore, $\varphi|_{S'}:S'\to W$ is a submersion, 
which factors as the composition
\[
\begin{tikzcd}S'\ar[r,"i"] &i(S')\ar[r,"\pi_2"] &W\end{tikzcd}.
\]
This implies that $T_{y_n}i(S')$ and $T_{y_n}(S'\times T')$ are transverse in $\big(T_{\pi_1(y_n)}S'\big)\times \mathbb R^m$. Therefore, $E_1$ and $E_2\times E_3$ are also transverse in $E_2\times\mathbb R^m$. Thus, $\dim E=\dim E_1\cap (E_2\times E_3)$ and hence $E=E_1\cap (E_2\times E_3)$.
\end{proof}

\begin{theorem}\label{ThomGeneral}
Consider the complex analytic germ 
$$\Phi=(g,f):(X,x)\to(\mathbb{C}^2,0)$$ 
and take a representative $X$ of $(X,x)$ such that it is a closed analytic subset of some open set $U\subseteq \mathbb{C}^N$, and consider also a Whitney stratification $\mathcal{S}$ of $X$. Assume that:
\begin{enumerate}[label=\arabic*., ref=\arabic*]
	\item $g$ is the restriction of a submersion $\tilde{g}:U\to\mathbb{C}$;\label[assumption]{item:submersion}
	\item $f^{-1}(0)$ is union of strata;\label[assumption]{item:f0f0}
	\item $\Gamma\coloneqq \overline{C(\Phi)\setminus f^{-1}(0)}$ is empty or a curve (i.e., it has dimension one), for $C(\Phi)$ the critical locus of $\Phi$;
	\item $\Phi^{-1}(0)\cap \Gamma\subseteq \left\{x\right\}$;\label[assumption]{item:finiteness}
	\item for each stratum $S\in\mathcal{S}$ such that $\dim S \geq1$, $\ker D_x\tilde{g}\notin \nu_{\bar{S},f}^{-1}(x)$, where $\nu_{\bar{S},f}:C(\bar{S},f)\to\bar{S}$ is given by the relative conormal bundle of $\bar{S}$.\label[assumption]{item:transverse}
\end{enumerate}
Then, with a possibly smaller representative, $\Phi$ is a Thom map with a stratification
$\left\{\mathcal{S}',\mathcal{T}\right\}$ 
such that $\mathcal{S}'$ refines $\mathcal{S}$.
\end{theorem}
\begin{proof}
First of all, observe that $f$ is a Thom map with the provided stratification $\mathcal{S}$ of $X$, together with the stratification $\left\{\mathbb{C}\setminus0,0\right\}$ in the target, by \cite[Theorem 4.2.1]{Briancon1994}. Now, we want a refinement of $\mathcal{S}$, say $\mathcal{S}'$, that makes $\Phi$ a Thom map with a convenient stratification $\mathcal{T}$ in the target. 

As we show below, this is attained if we refine $\mathcal{S}$ so that $\Gamma$, $f^{-1}(0)\setminus\Gamma$ and $X\setminus\big(f^{-1}(0)\cup \Gamma\big)$ are union of strata and the refinement is Whitney regular. We can use \cref{chungo} to achieve this on $X\setminus\big(f^{-1}(0)\cup \Gamma\big)$, with a restriction of $\Phi$ and the stratification $\mathcal{T}$ on $\mathbb{C}^2$ given by $0$, $\Delta\setminus0$, $L\setminus0$ and $\mathbb{C}^2\setminus\big(\Delta\cup L \big)$, where $\Delta\coloneqq\Phi(\Gamma)$ and $L\coloneqq \Phi\big(f^{-1}(0)\big)$. Observe that we already have that $f^{-1}(0)$ is a union of strata, by \cref{item:f0f0}, so we only need to check that adding $\Gamma\setminus x$ as stratum does not change the Whitney condition. However, as $\Gamma$ is a curve, the set where the Whitney condition could fail (sometimes called \textit{bad set}) is of dimension zero (see, for example, \cite[Proposition 2.6]{Gibson1976}). This implies that we can take a smaller representative of $(X,x)$ if needed where the result holds.

Now we show that the Thom condition holds for $\Phi$ and the stratifications $\mathcal{S}'$ and $\mathcal{T}$ as above. We consider two strata $S_r,S_t\in\mathcal{S}'$ such that $S_r\subseteq \bar{S_t}$ and a sequence $\left\{q_n\right\}_n\subset S_t$ converging to $p\in S_r$. We want to check Thom's regularity condition given in \cref{eq:Thomregcondition} for $\Phi$. To do so, we separate by cases: considering strata in $\Gamma$, $f^{-1}(0)\setminus\Gamma$ or $X\setminus\big( f^{-1}(0)\cup\Gamma\big)$.

If $S_r$ is contained in $\Gamma$, the Thom condition is always satisfied trivially because $\Phi$ is a local diffeomorphism when restricted to $\Gamma$ except, perhaps, at $x$. This comes from the fact that there is a representative of the restriction $\Phi|_\Gamma$ that is finite, by \cref{item:finiteness}, so $\Delta=\Phi(\Gamma)$ is a curve or empty.


If both $S_r$ and $S_t$ are contained in $f^{-1}(0)\setminus\Gamma$, we have that the Thom condition (recall \cref{eq:Thomregcondition}) for $\Phi$ is equivalent to
\begin{equation}\label{eq:thomf0}
T_p S_r \cap   \ker D_{p} \tilde{g} \subseteq  \lim_n \big( T_{q_n} S_t \cap \ker D_{q_n} \tilde{g}\big).
\end{equation}
However, as $\ker D_{\bullet} \tilde{g}$ is a hyperplane (by \cref{item:submersion}) and by \cref{item:transverse}, we have that
$$ \lim_n \big( T_{q_n} S_t \cap \ker D_{q_n} \tilde{g}\big) = \lim_n \big( T_{q_n} S_t \big)  \cap \ker D_{p} \tilde{g}$$
and \cref{eq:thomf0} is satisfied, as $\mathcal{S}$ is a Whitney stratification and $$\lim_n \big( T_{q_n} S_t \big)\supseteq T_p S_r.$$

In a similar fashion, if $S_t$ is contained in $X\setminus\big( f^{-1}(0)\cup\Gamma\big)$ and $S_r$ is contained in $X\setminus\big( f^{-1}(0)\cup\Gamma\big)$ or $f^{-1}(0)\setminus\Gamma$, we have that
\begin{align*}
	\ker D_pf|_{S_r} \cap  \ker D_{p} \tilde{g}\subseteq& \lim_n\ker D_{q_n}f|_{S_t} \cap  \ker D_{p} \tilde{g}\\
	=&\lim_n\big(D_{q_n}f|_{S_t} \cap  \ker D_{p} \tilde{g} \big),
\end{align*}
where the first inclusion is given by the Thom condition on $f$ and the second equality is given by \cref{item:submersion,item:transverse}.
%
\end{proof}

\begin{remark}\label{rem:Thomfl}
Finally, observe that we can apply \cref{ThomGeneral} with $$\Phi=(\ell,f):(X,x)\to(\mathbb{C}^2,0)$$ given in \cref{polar}, provided $\ell$ is generic enough, to make $\Phi$ a Thom map with certain stratification. More precisely, the objects $\Gamma$ and $\Delta$ of \cref{ThomGeneral} are, in this case, the polar curves and the Cerf's diagram of $f$ relatively to $\ell$.
\end{remark}

\begin{example}\label{ex:Sabbah}
In \cite[p. 286]{Sabbah1983} it is shown that the germ
\begin{align*}
	\Phi=(g,f):(\mathbb{C}^3,0)&\to(\mathbb{C}^2,0)\\
	(x,y,z)&\mapsto (y,x^2-y^2z)
\end{align*}
does not have a representative with the Thom condition. Indeed, it is \cref{item:transverse} from \cref{ThomGeneral} what fails (observe that, in this case, $\Gamma=\varnothing$), showing its importance.
\end{example}

\section{Lifting vector fields}

The construction of the Milnor fibration of a complex analytic function $f:(X,x)\to(\C,0)$ when $X$ is 
a complex analytic space is a consequence of the Thom-Mather first isotopy lemma. The strategy 
to prove \cref{nofixedpoints} is to take a generic linear form $\ell$ on the ambient space of 
$(X,x)$ and consider the map $\Phi=(\ell,f):(X,x)\to(\C^2,0)$. We want to trivialize this map in such 
a way that its composition with the projection onto the second component $\pi_2:(\C^2,0)\to(\C,0)$ 
gives the Milnor fibration of $f$ and its restriction to $\big(X\cap \ell^{-1}(0),x\big)$ gives the Milnor fibration 
of the restriction $f:\big(X\cap \ell^{-1}(0),x\big)\to(\C,0)$. This would allow us to use an induction process, 
as in \cite{Trang1975}. 

One could think that this could be done just by using the Thom-Mather second isotopy lemma. 
Unfortunately, this seems not possible and we are forced to use some of the ingredients in the 
proof of the isotopy lemmas, like controlled tube systems or controlled stratified vector fields, 
in order to construct a lifting of the vector field which fits into our problem. For the sake of
 completeness, we include in this section all the definitions and main results that we need for 
 that purpose. Instead of the original proof of the isotopy lemmas by Mather \cite{Mather2012}, 
 we follow the notations and statements of \cite[Chapter II]{Gibson1976}, where the reader can 
 find more details and the proofs of all the results.

%

We recall that a \emph{stratified vector field} on a stratified set $X$ of a smooth manifold $N$ is a map 
$\xi:X\to TN$ tangent to each stratum $S_\alpha$ of $X$ and smooth on $S_\alpha$, 
but $\xi$ might not be continuous. We now give the definitions of controlled tube system and controlled 
stratified vector field:


\begin{definition}[{{\it cf.} \cite[II.1.4]{Gibson1976}}]
If $X$ is a submanifold of $N$, a \textit{tube at $X$} is a quadruple $T=(E,\pi,\rho,e)$ where $\pi:E\to X$ is a smooth vector bundle, 
$\rho:E\to \mathbb{R}$ is a quadratic function of a Riemann metric on $E$ that vanishes on the zero section and $e:\big(E, \zeta(X)\big)\to (N, X)$ is a germ along $\zeta(X)$ 
of a local diffeomorphism, commuting with the zero section $\zeta:X\to E$ so that $e\circ\zeta$  along $X$ is the inclusion $X\subset N$.

If $X$ is a Whitney stratified subset of a manifold $N$, a \textit{tube system for the stratification} consists of a tube for every strata.
\end{definition}



\begin{definition}[{{\it cf.} \cite[II.2.5]{Gibson1976}}]\label{def:tube}
A tube system $\mathcal{T}=\left\{T_\alpha\right\}_{\alpha\in A}$, with $T_\alpha=\left(E_\alpha,\pi_\alpha,\rho_\alpha,e_\alpha\right)$, for a Whitney stratification $\left\{S_\alpha\right\}_{\alpha\in A}$ of some subset $X$ of a manifold $N$ is {\it weakly controlled} if the relation
$$ (\pi_\alpha\circ e_\alpha^{-1}) \circ (\pi_{\alpha'}\circ e_{\alpha'}^{-1})=(\pi_{\alpha}\circ e_{\alpha}^{-1}), \;\;\;\;\; \alpha,\alpha'\in A  $$
holds for every pair of tubes $(T_\alpha,T_{\alpha'})$ where the composition makes sense.
\end{definition}


We remark that the notion of weakly controlled tube system for a stratification is not a strange thing to ask, actually any given Whitney stratification admits a weakly controlled tube system (\textit{cf.} \cite[II.2.7]{Gibson1976}).



\begin{definition}[{{\it cf.} \cite[II.3.1]{Gibson1976}}]
If we have a tube system for a Whitney stratification of $X$ and $\xi$ is a stratified vector field on $X$ we shall say that $\xi$ is 
{\it a weakly controlled vector field} if:
$$ D(\pi_\alpha\circ e_\alpha^{-1})\circ\xi=\xi\circ (\pi_\alpha\circ e_\alpha^{-1}) $$
holds for every tube, using the notation of \cref{def:tube}.
\end{definition}

Next, we give the control conditions relative to a stratified map (recall \cref{def:stratmap}).


\begin{definition}[{{\it cf.} \cite[II.2.5]{Gibson1976}}]\label{controlled5}
Let $f:N\to N'$ be a smooth map and $X\subset N$ and $X'\subset N'$ two stratified sets such that $f(X)\subset X'$ and the induced map $\left.f\right|:X\to X'$ is a stratified map. Assume also that we have a tube system 
$\mathcal{T}=\left\{T_\alpha\right\}_{\alpha\in A}$ for the stratification $\left\{S_\alpha\right\}_{\alpha\in A}$ of $X$ and a tube system $\mathcal{T}'=\{T'_\beta\}_{\beta\in B}$ for  the stratification $\{S'_\beta\}_{\beta\in B}$ of $X'$. 
Then, we say that $\mathcal{T}$ is \emph{controlled} over $\mathcal{T}'$ if 
\begin{enumerate}[label=(\roman*), font=\itshape]
	\item $\mathcal{T}$ is weakly controlled,
	\item $f\circ (\pi_\alpha\circ e_\alpha^{-1}) = (\pi_\beta\circ e_\beta^{-1}) \circ f$, for every $S_\alpha$ mapping into $S_\beta '$, and
	\item $ (\rho_\alpha\circ e_\alpha^{-1})\circ (\pi_{\alpha'}\circ e_{\alpha'}^{-1})=(\rho_\alpha\circ e_\alpha^{-1})$ holds for every pair 
	$(T_\alpha,T_{\alpha'})$ such that $f(S_\alpha\cup S_{\alpha'})\subseteq S_\beta '$ for some $S_\beta'$ in $X'$.
\end{enumerate}
\end{definition}

In addition to the control conditions we also need a regularity condition for stratified maps of the same nature as the Whitney condition for stratified sets. This is the Thom condition given in \cref{Thomconditiondef}. In fact, the Thom condition ensures the existence of a controlled tube system as follows:

\begin{theorem}[{{\it cf.} \cite[II.2.6]{Gibson1976}}]\label{vf2}
Let $N,N',X, X',f$ be as in \cref{controlled5} and assume $\left.f\right|:X\to X'$ is a Thom map. Then, each weakly controlled tube system $\mathcal{T}'$ of $X'$ has a tube system $\mathcal{T}$ of $X$ controlled over $\mathcal{T}'$.
\end{theorem}


We also have control conditions relative to a stratified map for stratified vector fields.

\begin{definition}[{{\it cf.} \cite[II.3.1]{Gibson1976}}]
Let $X,X',\mathcal{T}$,$\mathcal{T}'$, and $\left.f\right|:X\to X'$ be as in \cref{controlled5}. Assume that we have $\xi$ and $\xi'$ stratified 
vector fields on $X$ and $X'$, respectively. We say that $\xi$ \textit{lifts} $\xi'$ if $D f|_{S_\alpha}\circ \xi = \xi'\circ f|_{S_\alpha}$, for every stratum $S_\alpha$ of $X$. Furthermore,
 we say that $\xi$ is \emph{controlled} over $\xi'$ if
\begin{enumerate}[label=(\roman*), font=\itshape]
  \item $\xi$ lifts $\xi'$,
  \item $\xi$ is weakly controlled,
  \item $D(\rho_\alpha\circ e_\alpha^{-1})\circ\left.\xi\right|_{f^{-1}S_\beta'}=0 $	for every  $S_\alpha$ mapping into $S_\beta '$.
	%
	\end{enumerate}
\end{definition}


Again, the Thom condition is the key point to lift any weakly controlled vector field in the target to a controlled vector field in the source:

\begin{theorem}[{{\it cf.} \cite[II.3.2]{Gibson1976}}]\label{vf1}
Let $N,N',X, X',f$ be as in \cref{controlled5} and assume $\left.f\right|:X\to X'$ is a Thom map. Let $\mathcal{T}$ and $\mathcal{T}'$ be tube 
systems of the stratifications of $X$ and $X'$, respectively, such that $\mathcal{T}$ is controlled over $\mathcal{T}'$. 
Then, any weakly controlled vector field $\xi'$ on $X'$ lifts to a stratified vector field $\xi$ which is controlled over $\xi'$.
\end{theorem}


The last ingredient is about integrability of stratified vector fields. Specifically, if we have a stratified vector field $\xi$ on $X$ and we integrate it on 
every stratum $S_\alpha$ we have a smooth flow $\theta_\alpha:D_\alpha\to S_\alpha$, where $D_\alpha\subseteq \mathbb{R}\times S_\alpha$ 
is the maximal domain of the integration, which contains $\{0\}\times S_\alpha$. Setting $D$ as the union of every $D_\alpha$, 
we obtain a map $\theta:D\to X$ that is not necessarily continuous. 

\begin{definition}[{\textit{cf.} \cite[II.4.3]{Gibson1976}}]%
With the notation above, if $\theta$ is continuous on a neighbourhood of $\left\{0\right\}\times X$ we say that $\xi$ is \emph{locally integrable}. 
Furthermore, if $D=\mathbb{R}\times X$ we say that $\xi$ is \emph{globally integrable}.
\end{definition}

It is here where the control conditions over the vector fields play their role:

\begin{theorem}[{{\it cf.} \cite[II.4.6]{Gibson1976}}]\label{vf3}
Let $N,N',X, X',f$ be as in \cref{controlled5}. Assume also that $X$ is locally closed in $N$. 
If $\xi$ and $\xi'$ are stratified vector fields on $X$ and $X'$, respectively, and $\xi$ is controlled over $\xi'$ with respect to some 
tube system $\mathcal{T}$ of $X$, then $\xi$ is locally integrable if $\xi'$ is so.
\end{theorem}

\begin{theorem}[{{\it cf.} \cite[II.4.8]{Gibson1976}}]\label{vf4}
Let $N,N',X, X',f$ be as in \cref{controlled5}. Assume also $f|:X\to X'$ is proper. If $\xi$ and $\xi'$ are stratified vector fields on $X$ and $X'$, 
respectively, and $\xi$ is locally integrable, then $\xi$ is globally integrable if $\xi'$ is so.
\end{theorem}

So, to summarize, if we combine \cref{vf1,vf2,vf3,vf4} we get:

\begin{corollary}\label{lifting}
Let $N,N',X, X',f$ be as in \cref{controlled5}  and assume $f|:X\to X'$ is a Thom proper map. If we have a weakly 
controlled tube system with a weakly controlled vector field $\xi'$ on $X'$ such that it is globally integrable, it lifts to a globally 
integrable vector field on $X$.
\end{corollary}

\begin{corollary}\label{lifting2}
Let $f:N\to N'$ be a smooth map and let $X\subset N$ be a Whitney stratified subset such that $f|:X\to N'$ is a proper stratified submersion. If $\xi'$ is a globally integrable smooth vector field on $N'$, it lifts to a globally 
integrable vector field on $X$.
\end{corollary}

\cref{lifting2} is a consequence of \cref{lifting} since the Thom condition is satisfied in this case (see \cite[II.3.3]{Gibson1976}).
We also remark that these two corollaries, among other things, are used in \cite{Gibson1976} to prove the Thom-Mather isotopy lemmas.

Finally, we show how \cref{lifting2} can be used to construct a local geometric monodromy of a function in a specific way.
Let $\varphi\colon X\to \mathbb{S}^1$ be a locally trivial $C^0$- fibration with fiber $F=\varphi^{-1}(t_0)$. It is well known 
that $\varphi\colon X\to \mathbb{S}^1$ is $C^0$- equivalent to the fibration 
$\pi\colon\nicefrac{F\times \left[0,2\pi\right]}{\sim}\to \mathbb{S}^1$, where the relation is given 
by $(x,0)\sim\big(h(x),2\pi\big)$ for some homeomorphism $h:F\to F$ and $\pi([x,t])=t$.  As we have mentioned in the introduction,
such homeomorphism $h$ is called a 
\textit{geometric monodromy} of $\varphi$, although it is not smooth in general. Since there are some choices,
a geometric monodromy is not unique. However one can prove that its isotopy 
class is well defined, so the induced map on homology (or cohomology) is uniquely given by $\varphi$ and it is simply  called 
\textit{the monodromy} of $\varphi$.

In our case, given a complex analytic function $f\colon(X,x)\to(\C,0)$ there exist $\epsilon$ and $\eta$ with $0<\eta\ll\epsilon\ll1$ such that
\begin{equation}\label{fibration}
f\colon X\cap B_\epsilon\cap f^{-1}(\partial D_\eta)\to \partial D_\eta
\end{equation}
is a proper stratified submersion, for some Whitney stratification on the source and the trivial stratification on $\partial D_\eta$ 
(see \cite{Trang1976}).
By the Thom-Mather first isotopy lemma, \eqref{fibration} is a locally trivial $C^0$- fibration with fiber $F$. 

In fact, we have something more. We take on $\partial D_\eta$ the vector field of constant length $\eta$ 
and positive direction. By \cref{lifting2}, this vector field can be lifted to a stratified vector field $\xi$ 
on the source which is globally integrable. 

The flow of $\xi$ provides the local trivialisations of \eqref{fibration} 
and it follows that the geometric monodromy obtained in this way $h:F\to F$ is a stratified homeomorphism 
(that is, it preserves strata and the restriction on each stratum is a diffeomorphism). We call $h:F\to F$ the 
\emph{local geometric monodromy of $f$ at $x$ induced by $\xi$}.

In the next section we show that instead of a Euclidean ball $B_\epsilon$ we can take a convenient polydisc, 
which is better to proceed with the induction hypothesis.

\section{Privileged polydiscs}

In \cite{Trang1975}, instead of a usual Milnor ball $B$ for a complex analytic function $f:(\C^{n+1},x)\to(\C,0)$, it is considered a privileged polydisc $\Delta=D_1\times\dots\times D_{n+1}$ with respect to some generic choice of coordinates $z_1,\dots,z_{n+1}$ in $\C^{n+1}$. Here we show how to adapt this notion to the case of a function $f:(X,x)\to(\C,0)$ on a complex analytic set $X$.

Assume that $\dim(X,x)=n+1$ and that $(X,x)$ is embedded in $(\C^N,0)$. 
We take a representative $f:X\to \C$ and complex analytic Whitney stratifications in $X$ and 
$\mathbb{C}$ such that $f:X\to \C$ is a stratified function. 
We say that $z_1,\dots,z_N$ are \emph{generic coordinates} if for each $i=0,\dots,n$, the $(N-i)$-plane $H^i$ through the origin given by $\{z_1=\dots=z_i=0\}$ is transverse to all the strata of $X$ except, perhaps, the stratum $\{x\}$. 

We consider the set $X^i=\pi_i(X\cap H^i)\subset\C^{N-i}$, where 
$\pi_i$ is the projection onto the last $N-i$ coordinates, with the induced stratification and the function $f^i: X^i\to\C$ given by
\[
f^i(z_{i+1},\dots,z_N)=f(0,\dots,0,z_{i+1},\dots,z_N).
\]

A \emph{polydisc} centered at $0$ in $\C^N$ is a set of the form $\Delta=D_1\times\dots\times D_n\times B$ 
where $D_1,\dots,D_n\subset\C$ are  discs and $B\subset\C^{N-n}$ is a ball centered at $x$. We also denote by $\Delta^i=D_{i+1}\times\dots\times D_n\times B$ the corresponding polydisc in $\C^{N-i}$. Each polydisc $\Delta^i$ is considered with the obvious Whitney stratification given by taking all combinations of products of interiors and boundaries on the discs and the ball (see \cite[1.3]{Trang1975}). 

Observe that a polydisc has a ball in the product of its definition. This ball is necessary to have control on the codimension of $X$. More precisely, as we will work with the sets $X^i$ and $X^i\cap\Delta^i$, we need to stop taking sections as soon as $X^i$ is a curve and, at that point, we take a ball that completes the product structure we want to find (the so-called polydiscs).

\begin{definition} We say that $\Delta$ is a \emph{privileged} polydisc if for any smaller polydisc $\Delta'\subset \Delta$ centered at $0$ in $\C^N$, all the strata of $\big(\Delta'\big)^i$ are transverse to all the strata of $X^i$, for all $i=0,\dots,n$.
\end{definition}

For each privileged polydisc  $\Delta$, the set $X^i\cap\Delta^i$ has an induced Whitney stratification. By the curve selection lemma, the function $f^i:X^i\cap\Delta^i\to\C$ has an isolated critical value in the stratified sense at the origin in $\C$. So, $(f^i)^{-1}(b)$ is transverse to all the strata of $X^i\cap\Delta^i$, for all $b\in\C$ small enough. In particular, there exists $\eta>0$ small enough such that
\[
f^i: X^i\cap\Delta^i\cap (f^i)^{-1}(\partial D_\eta)\longrightarrow\partial D_\eta
\]
is a proper stratified submersion and hence, a locally $C^0$-trivial fibration homotopic to a Milnor fibration with a homotopy which preserves the fibres. This follows from the Thom-Mather first isotopy lemma and the fact that privileged polydiscs are good neighbourhoods relatively to $\left\{f=0\right\}$ in Prill's sense (cf. \cite{Prill1967}), see the end of \cite[Section 1]{Trang1975} for more details. In fact, this is the original definition of privileged polydisc in \cite{Trang1975} in the case $X=\C^{n+1}$. The existence of privileged polydiscs is proved in the next lemma:

\begin{lemma}\label{lem:polydisc}  Any small enough polydisc is privileged.
\end{lemma}
\begin{proof} We show by induction on $i=0,\dots,n$ that $f^i$ has a privileged polydisc $\Delta^i$. The case $i=n$ is obvious since a privileged polydisc is nothing but a Milnor ball. Assume $f^i$ has a privileged polydisc $\Delta^i$. We shall find a disc $D_\epsilon$ such that $D_\epsilon\times\Delta^i$ is a privileged polydisc for $f^{i-1}$. We use the function $\rho:\C^{N-i+1}\to \mathbb R$ given by $\rho(z)=|z_i|$.
By the curve selection lemma we can find $\epsilon>0$ such that for any $0<\epsilon'\leq\epsilon$, 
$\partial D_{\epsilon'}\times\mathbb{C}^{N-i}$ is transverse to each stratum of $X^{i-1}$.

Consider the polydisc $D_{\epsilon'}\times\left(\Delta^i\right)'$, for a polydisc $\left(\Delta^i\right)'$ 
contained in $\Delta^i$ and $\epsilon'\leq\epsilon$. We have two types of strata: $\mathring{D}_{\epsilon'}\times R_\alpha$ and $\partial D_{\epsilon'}\times R_\alpha$, for some stratum $R_\alpha$ 
of $\left(\Delta^i\right)'$. On the other hand, if we consider a stratum $S_\beta$ of $X^{i-1}$,  and we take the hyperplane section to get $X^i$, it gives the stratum $S_\beta'$ of $X^i$.

By induction hypothesis, $R_\alpha$ is transverse to $S_\beta'$, that is,
\begin{equation}\label{transv}
T_z R_\alpha + T_z S_\beta'= \mathbb{C}^{N-i},
\end{equation}
for all $z\in R_\alpha\cap S_\beta'$. This obviously implies that
$$\C\times T_z R_\alpha +T_{(t,z)} S_\beta= \C\times \mathbb{C}^{N-i},$$
which gives the transversality between $\mathring{D}_{\epsilon'}\times R_\alpha$ and $S_\beta$ at $(t,z)$.

Moreover, the choice of $\epsilon$ implies that 
\[
T_t \partial D_{\epsilon'}\times\mathbb{C}^{N-i} + T_{(t,z)} S_\beta =\C\times\mathbb{C}^{N-i},
\]
for all $(t,z)\in (\partial D_{\epsilon'}\times \C^{N-i})\cap (V\times S_\beta)$. Therefore, any vector $(u,v)\in \C\times\mathbb{C}^{N-i}$ can be written as $(u,v)=(u_1,v_1)+(u_2,v_2)$, for some $(u_1,v_1)\in  T_t \partial D_{\epsilon'}\times\mathbb{C}^{N-i}$ 
and $(u_2,v_2)\in T_{(t,z)} S_\beta$. If $z\in R_\alpha$, 
we also have by \eqref{transv} that $v_1=w_1+w_2$, with $w_1\in T_z R_\alpha$ and $w_2\in T_z S_\beta'$. We get
\[
(u,v)=(u_1,w_1)+(0,w_2)+(u_2,v_2),
\]
with $(u_1,w_1)\in T_t \partial D_{\epsilon'}\times T_z R_\alpha$ and $(0,w_2)+(u_2,v_2)\in T_{(t,z)} S_\beta$. This shows that $\partial D_{\epsilon'}\times R_\alpha$ is also transverse to $S_\beta$.
\end{proof}

\begin{remark}
We see in the proof of \cref{lem:polydisc} that the choice of the radius of each disc of $\Delta$ is independent of the radii of the other discs. The reason of this independence is that we were asking that $\partial D_{\epsilon'}\times\mathbb{C}^{N-i}$ has to be transverse to each stratum of $X^{i-1}$ at any point instead of being transverse only at points on $X^{i-1}\cap \mathbb{C}\times \Delta^i$, which would have given $D_{\varepsilon'}$ a relation with $\Delta^i$ that restricts it. However, as presented in the proof, there could be a relation between the radii of the discs and the radius of the ball.\end{remark}

\section{Proof of the main theorem}
%

In this section we give the proof of \cref{nofixedpoints}. The proof is by induction on the dimension of $(X,x)$. To do this, 
we need the carrousel construction in \cite{Trang1975} of the second named author. We also refer to \cite{LNS} for a detailed 
construction of the carrousel for a general complex analytic germ of plane curve $(C,0)$. 
In our case, we apply this construction for the Cerf's diagram 
$C=\Delta_\ell(f,x)$ of a holomorphic function $f:(X,x)\to(\C,0)$ with respect to a generic linear form $\ell$. The key point 
here is that if $f\in\mathfrak m_{X,x}^2$, then all the branches of $C$ are tangent to the axis $\{v=0\}$ at the origin, where $u,v$ are the 
coordinates of the plane $\C^2$ (see \cref{prop tangent}).

\begin{lemma}[{{\it cf.} \cite[3.2.2]{Trang1975}}]\label{carrousel} Let $(C,0)$ be a germ of complex analytic plane curve whose tangent cone is the axis $\{v=0\}$.
There exist small enough discs $D$ and $D_\eta$ centered at the origin in $\C$ and a smooth vector field $\omega$ on the solid torus 
$D\times \partial D_\eta$ such that:
\begin{enumerate}[label=(\roman*)]
	\item \label{carrouseli}The projection onto the second component of $\omega$ gives the unit tangent vector field over
	$\partial D_\eta$ (\textit{i.e.}, the tangent field of length $\eta$, the radius of $D_\eta$, in the positive direction),
	\item \label{carrouselii} the restriction to $\left\{0\right\}\times\partial D_\eta$ is indeed the unit vector field,
	\item \label{carrouseliii}the vector field $\omega$ is 
	tangent to $(D\times \partial D_\eta)\cap\{\delta=\epsilon\}$ for all $\epsilon\in\C$ small enough, where $\delta=0$ is a reduced equation of $C$, and
	\item \label{carrouseliv}the only integral curve that is closed after a loop in $\partial D_\eta$ is $\left\{0\right\}\times\partial D_\eta$.
\end{enumerate}
\end{lemma}

The discs $D$ and $D_\eta$ in \cref{carrousel} are chosen small enough so that there is a disc $D_1$ containing $D$ strictly such that 
$\big(D_1\times \{0\}\big)\cap C=\{0\}$ and such that $\{v=t\}$, 
for $\eta\geq |t|>0$, intersects the curve $C$ in $\big(D\times \{0\}, C\big)_0$ points
in $D\times D_\eta$ where $(\bullet,\bullet)_0$  is the local intersection number at $0$ (see \cref{fig:Cerf}).

\begin{figure}[ht]
	\centering
		\includegraphics[scale=0.7]{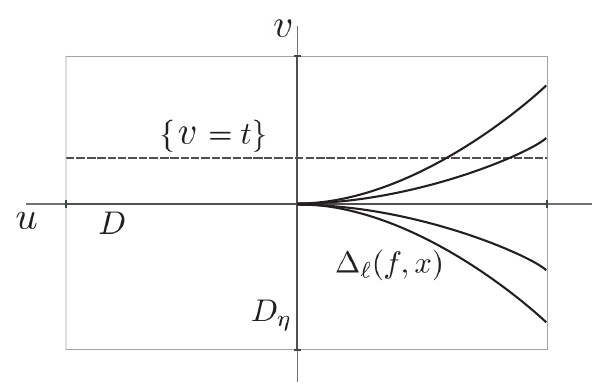}
	\caption{Cerf's diagram and the setting to construct the carrousel.}
	\label{fig:Cerf}
\end{figure}

The geometrical meaning of the carrousel is the following (see \cref{fig: carrousel}): we first take a representative $C$ of the plane curve on some open neighbourhood $W$ of the origin in the plane 
$\C^2$. Let $L$ be the intersection of $W$ with the axis $\{u=0\}$. We consider $W$ with the Whitney stratification given by the strata 
$W\setminus(C\cup L)$, $C\setminus\{0\}$, $L\setminus\{0\}$ and $\{0\}$ and the function germ 
$\pi_2:(W,0)\to(\C,0)$ given by $\pi_2(u,v)=v$. The choice of $D$ and $D_\eta$ is made so that 
\[
\pi_2: D\times \partial D_\eta\to \partial D_\eta
\]
is a proper stratified submersion with the induced stratification in $D\times \partial D_\eta$. By \cref{carrouseli,carrouselii,carrouseliii} in \cref{carrousel}, $\omega$  is a stratified vector field on $D\times \partial D_\eta$ which is a lifting of the unit tangent vector field on $\partial D_\eta$. Hence, its flow provides a local geometric monodromy $h:D\times\{t\}\to D\times\{t\}$ for some $t\in\partial D_\eta$, which preserves the point $(0,t)$ and the finite set $C\cap \big(D\times\{t\}\big)$. By condition (iv), the only fixed point of $h$ is $(0,t)$.

\begin{figure}[htb]
	\centering
		\includegraphics[width=0.95\textwidth]{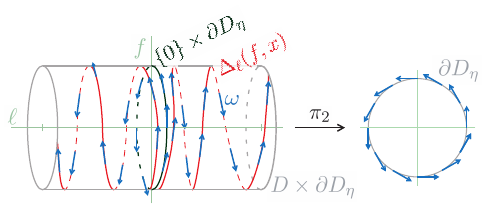}
	\caption{Representation of a carrousel $\omega$.}
	\label{fig: carrousel}
\end{figure}

\medskip
Now we can give the proof of our main result:

\begin{proof}[Proof of \cref{nofixedpoints}]
Assume that $(X,x)\subset (\C^N,x)$. We take a privileged polydisc $\Delta$ in $\C^N$ at $x$ and a small 
disc $D_\eta$ in $\C$ at $0$ such that the restriction
\[
f:X\cap \Delta\cap f^{-1}(\partial D_\eta)\to \partial D_\eta
\]
is a proper stratified submersion. We claim that there exists a stratified vector field $\xi$ on $X\cap\Delta\cap f^{-1}(\partial D_\eta)$ 
which is a lifting of the unit vector field $\theta$ on $\partial D_\eta$ whose flow provides a local geometric 
monodromy with no fixed points. We prove this by induction on the dimension of $X$ at $x$.

Assume first that $\dim(X,x)=1$. Let $X_1,\dots,X_r$ be the analytic branches of $X$ at $x$. 
Then $X\cap \Delta\cap f^{-1}(\partial D_\eta)$ is the disjoint union of all the sets 
$X_i\cap \Delta\cap f^{-1}(\partial D_\eta)$, $i=1,\dots,r$. Hence, it is enough to show the claim 
in the case that $X$ is irreducible at $x$. Let $n:\tilde X\to X$ be the normalization of $X$ at $x$.
Since $f\in\mathfrak m_{X,x}^2$, we can take an analytic extension $\overline f:(\C^N,x)\to(\C,0)$ 
such that $F\in\mathfrak m_N^2$. After a reparametrization, we can assume that 
$\tilde X$ is an open neighbourhood of $0$ in $\C$, $\left\{0\right\}=n^{-1}(x)$ and $F\circ n(s)=s^{k}$, for some $k\ge 2$. In this case, $\theta$ lifts in a unique way by the map $F\circ n$ and has a local geometric monodromy with no fixed points. But $n$ induces a diffeomorphism on $\tilde X\setminus\{0\}$ onto $X\setminus\left\{x\right\}$, so we have also a unique lifting on 
$X\cap \Delta\cap f^{-1}(\partial D_\eta)$ whose geometric monodromy has no fixed points. 

Now we assume the claim is true when $\dim (X,x)=n$ and prove it in the case that $\dim (X,x)=n+1$. 
Let $\ell:\C^N\to\C$ be a generic linear form and consider the map $\Phi=(\ell,f)$. We have a commutative diagram as follows:
\[
\begin{tikzcd}X\cap \Delta\cap\Phi^{-1}(D\times\partial D_\eta)\ar[r,"\Phi"] &D\times\partial D_\eta\ar[r,"\pi_2"]&\partial D_\eta\\
X\cap \Delta\cap \ell^{-1}(0)\cap f^{-1}(\partial D_\eta)\ar[r,"{(0,f)}"]\ar[u,hook] &\{0\}\times\partial D_\eta\ar[ru,"\pi_2"']\ar[u,hook]&\end{tikzcd},
\]
where the vertical arrows are the inclusions and $\pi_2$ is the projection onto the second component. 
Here we choose the polydiscs $\Delta$ and $D\times D_\eta$ small enough such that $\Phi$ is a Thom proper map (see \cref{lem:polydisc,ThomGeneral,rem:Thomfl}). 
The stratification in $D\times\partial D_\eta$ is given by the strata $D\times\partial D_\eta\setminus(C\cup L)$, 
$(D\times\partial D_\eta)\cap C$ and $L$, where $L=\{0\}\times \partial D_\eta$ and $C=\Delta_\ell(f,x)$ is the Cerf's diagram.

By induction hypothesis, there exists a stratified vector field $\xi_1$ on $X\cap\Delta\cap \ell^{-1}(0)\cap f^{-1}(\partial D_\eta)$ 
which is a lifting of $\theta$ and whose  geometric monodromy has no fixed points. If $C$ is empty, the claim 
is obvious by \cref{emptycerf}, so we can assume that $C$ is not empty. 

By the carrousel of \cref{carrousel}, there exists a stratified vector field $\omega$ on $D\times\partial D_\eta$ which 
satisfies \cref{carrouseli,carrouselii,carrouseliii,carrouseliv} of the lemma.  Since $\omega$ is a lifting of $\theta$,  it is globally integrable 
by \cref{vf4}. Moreover, $\omega$ is not zero along $L$ and $(D\times\partial D_\eta)\cap C$, so we can use 
the flow of $\omega$ to construct a weakly controlled tube system $\mathcal T'$ of $D\times\partial D_\eta$ such that 
$\omega$ is weakly controlled. By \cref{lifting}, $\omega$ lifts to a stratified vector field $\xi$ on 
$X\cap \Delta\cap\Phi^{-1}(D\times\partial D_\eta)$ which is globally integrable.
Moreover, by using a partition of unity, we can construct $\xi$ in such a way that it coincides with $\xi_1$ on 
$X\cap\Delta\cap \ell^{-1}(0)\cap f^{-1}(\partial D_\eta)$.

Let $F=X\cap\Delta \cap f^{-1}(t)$, with $t\in\partial D_\eta$ and consider the geometric monodromy $h:F\to F$ 
induced by $\xi$. On one hand, $\xi$ is an extension of $\xi_1$, so $h\big(F\cap\ell^{-1}(0)\big)=F\cap\ell^{-1}(0)$ and 
$h$ has no fixed points on $F\cap\ell^{-1}(0)$. On the other hand, \cref{carrouseliv} of \cref{carrousel} implies 
that $h$ does not have fixed points on $F\setminus\ell^{-1}(0)$ either. This completes the proof.
\end{proof}

%


%

%
%
%
%
%
%
 


The proof relied on the hypothesis of $f$ being in $\mathfrak{m}_{X,x}^2$, and actually this hypothesis is necessary. Here we give a couple of examples which illustrate this fact.

\begin{example}\label{triple:ex} Let $(C,0)$ be the ordinary triple point singularity in $(\C^3,0)$. This is equal to the union of the three coordinate axis in $\C^3$ and the defining equations are given by the $2\times 2$-minors of the matrix
\[
M=\left(
\begin{array}{ccc}
x  & y  & z  \\
y  & z  & x  
\end{array}
\right)
\]
This gives to $(C,0)$ a structure of isolated determinantal singularity in the sense of \cite{NOT}. According also to \cite{NOT}, we can construct a determinantal smoothing of $(C,0)$ by taking the $2\times 2$-minors of $M_t=M+tA$, where $A$ is a generic $2\times 3$-matrix with coefficients in $\C$ and $t\in\C$. 

In fact, let
\[
A=\left(
\begin{array}{ccc}
0  & 1  & 0  \\
0  & 0  & 0  
\end{array}
\right)
\]
and let $(X,0)$ be the surface in $(\C^3\times\C,0)$ defined as the zero set of the $2\times 2$-minors of $M_t$. The projection $f:(X,0)\to(\C,0)$, $f(x,y,z,t)=t$ provides a flat deformation whose special fibre is $(C,0)$ and whose generic fibre $F=f^{-1}(t)$, for $t\ne0$, is a smooth curve. We can see $F$ as a kind of ``determinantal Milnor fibre'' of $(C,0)$. 

It follows from \cite[page 279]{BG} that $F$ is diffeomorphic to a disk with two holes (as in \cref{triple}) and that the monodromy $h_*:H_1(F;\Z)\to H_1(F;\Z)$ is the identity. Since $H_1(F;\Z)\cong\Z^2$, the Lefschetz number is $-1$, and hence any local geometric monodromy must have a fixed point. A simple computation shows that $f\notin\mathfrak{m}_{X,x}^2$ in this example.

\begin{figure}[ht]
	\centering
		\includegraphics[width=0.85\textwidth]{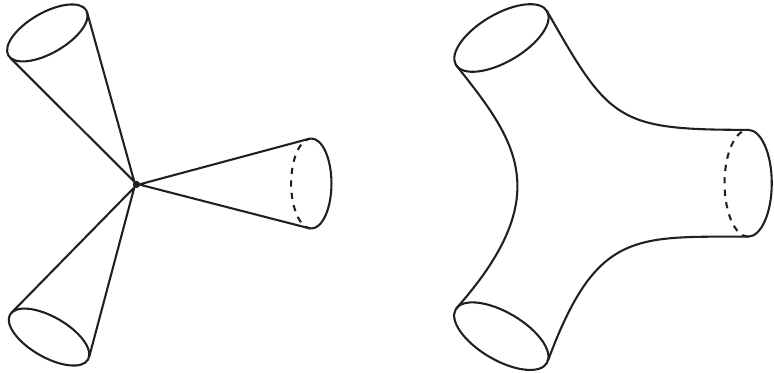}
		\caption{The ordinary triple point singularity and its determinantal Milnor fibre.}
	\label{triple}
\end{figure}
\end{example}

\begin{example} Consider the $A_4$ plane curve singularity $(C,0)$ whose equation in $(\C^2,0)$ is $x^5-y^2=0$. 
The monodromy of the classical Milnor fibre of $(C,0)$ is well known and we will not discuss it. Instead, we look 
at the monodromy of the disentanglement of $(C,0)$ in Mond's sense (see \cite[Chapter 7]{Mond-Nuno2020}). 
We see $(C,0)$ as the image of the map germ $g_0\colon(\C,0)\to(\C^2,0)$ given by $g_0(s)=(s^2,s^5)$, which 
has an isolated instability at the origin. 

Since we are in the range of Mather's nice dimensions, we can take a 
stabilisation, that is, a 1-parameter unfolding $G\colon (\C\times \C,0)\to(\C^2\times\C,0)$, $G(s,t)=(g_t(s),t)$ 
such that for any $t\ne0$, $g_t$ has only stable singularities. By definition, the \emph{disentanglement} $F$ 
is the image of the map $g_t$ intersected with a small enough ball $B$ in $\C^2$ centered at the origin 
and $t$ small enough. Since $F$ is 1-dimensional and connected, it has the homotopy type of a bouquet of spheres 
(this is true also in higher dimensions by a theorem due to Lê) of dimension 1. The number of such spheres is called the 
\emph{image Milnor number} and is denoted by $\mu_I(g_0)$. 

Observe that $F$ is also the generic fibre 
of the function $f:(X,0)\to(\C,0)$ where $(X,0)$ is the image of $G$ in $(\C^2\times\C,0)$ and $f(x,y,t)=t$. 
We are interested in the local monodromy of $f$ at the origin.

In our case, we take $g_t(s)=(s^2,s^5+ts)$. It is easy to see that for $t\ne0$, $g_t$ is an immersion with  
two transverse double points $p=g_t(a_1)=g_t(a_2)$ and $q=g_t(b_1)=g_t(b_2)$ where $a_1,a_2,b_1,b_2$ 
are the four roots of $s^4+t=0$, with $a_1=-a_2$ and $b_1=-b_2$. Hence, $g_t$ defines a stabilisation of $g_0$. 
Observe that the number of double points coincides with the delta invariant $\delta(C,0)=2$. 
The disentanglement $F$ is the image of $g_t$ and is homeomorphic to the quotient of a closed 2-disk $D_t$ 
under the relations $a_1\sim a_2$ and $b_1\sim b_2$ (see \cref{pic1}). Thus, $F$ has the homotopy 
type of $S^1\vee S^1$ and $\mu_I(g_0)=2$. 

\begin{figure}[ht]
	\centering
		\includegraphics[width=0.8\textwidth]{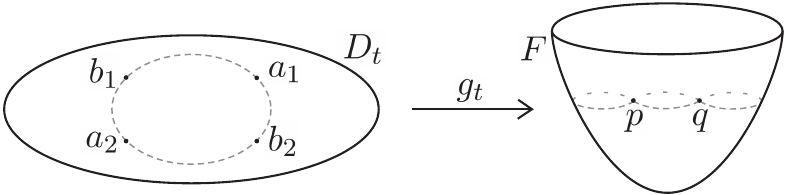}
		\caption{The map $g_t$ and the double points, $a_1,b_1,a_2$ and $b_2$.}
	\label{pic1}
\end{figure}

The locally $C^0$-trivial fibration is the restriction $f:X\cap(B\times S^1_\eta)\to S^1_\eta$, for a small enough $\eta>0$. 

In order to construct a geometric monodromy $h\colon F\to F$ it is enough to find a 1-parameter group of stratified homeomorphisms 
$h_\theta\colon X\cap(B\times S^1_\eta)\to X\cap(B\times S^1_\eta)$, with $\theta\in\mathbb R$, which make the following diagram commutative
\[
\begin{tikzcd}X\cap(B\times S^1_\eta) \ar[r,"f"]\ar[d,"h_\theta"']& S^1_\eta\ar[d,"r_\theta"]\\
X\cap(B\times S^1_\eta) \ar[r,"f"]& S^1_\eta\end{tikzcd},
\]
where $r_\theta(t)=e^{i\theta} t$. In this situation, $h:F\to F$ is obtained as the restriction of $h_{2\pi}$.

Since $(C,0)$ is weighted homogeneous with weights $(5,2)$, instead of a Euclidean ball in $\C^2$ it is better 
to consider the (non-Euclidean) ball $B$ given by $|x|^5+|y|^2\le 1$. Thus, $C\cap B=g_0(D)$, where $D$ 
is the disk in $\C$ given by $|s|^{10}\le 1/2$. For $t\ne0$, $F=g_t(D_t)$, where now $D_t=g_t^{-1}(B)$ is 
the disk in $\C$ given by 
\[
|s|^{10}+|s|^2|s^4+t|^2\le 1.
\]

Given a point $(x,y,t)\in X$, we have $(x,y,t)=G(s,t)$ for some $s\in\C$. We define $h_\theta\colon X\to X$ as
\[h_\theta(G(s,t))=G\left(e^{\frac{i\theta}{4}}s,e^{i\theta} t\right).
\]
Now, we have to check that, indeed, this gives a group of stratified homeomorphisms. We consider in $X$ the stratification given by $\{X\setminus Y,Y\}$, where $Y$ is the double point curve with equations $x^2+t=0$, $y=0$.
Since $G$ is an embedding on $X\setminus Y$, $h_\theta$ is well defined and is a diffeomorphism on $X\setminus Y$. When $(x,y,t)\in Y$ 
we have $G(s,t)=(s^2,0,t)$, with $s^2=x$ and $s^4+t=0$. It follows that 
\[
h_\theta(x,0,t)=G\left(e^{\frac{i\theta}{4}}s,e^{i\theta} t\right)=\left(e^{\frac{i\theta}{2}}s^2,0,e^{i\theta} t\right)=\left(e^{\frac{i\theta}{2}}x,0,e^{i\theta} t\right),
\]
and $\left(e^{i\theta/2}x\right)^2+e^{i\theta} t=e^{i\theta}(x^2+t)=0$. Thus, $h_\theta$ is also well defined on $Y$, $h(Y)=Y$ 
and the restriction $h\colon Y\to Y$ is a diffeomorphism. It is also clear that $h_\theta\colon X\to X$ and its inverse are both 
continuous, so it is a stratified homeomorphism.
It only remains to show that $h_\theta\big(X\cap(B\times S^1_\eta)\big)=X\cap(B\times S^1_\eta)$, because we have to work with a specific representative, but this a consequence of the equality:
\[
\left|e^{\frac{i\theta}{4}}s\right|^{10}+\left|e^{\frac{i\theta}{4}}s\right|^2\left|\left(e^{\frac{i\theta}{4}}s\right)^4+e^{i\theta}t\right|^2=
|s|^{10}+|s|^2|s^4+t|^2.
\]
The geometric monodromy $h\colon F\to F$ is now the restriction of $h_{2\pi}$, which gives $h(g_t(s))=g_t\left(e^{i\pi/2}s\right)$, 
that is, it is obtained by a $\pi/2$-rotation in the disk $D_t$.

To finish, we compute $h_*\colon H_1(F;\Z)\to H_1(F;\Z)$. We recall that $F$ is homeomorphic to the quotient of $D_t$ under 
the relations $a_1\sim a_2$ and $b_1\sim b_2$. The four points $a_1,a_2,b_1,b_2$ are on a square contained in the interior 
of $D_t$ and centered at the origin, which is obviously invariant under the $\pi/2$-rotation. We denote by $a,b,c,d$ the four 
edges of the square as in \cref{pic2}.

\begin{figure}[H]
	\centering
		\includegraphics[width=1.00\textwidth]{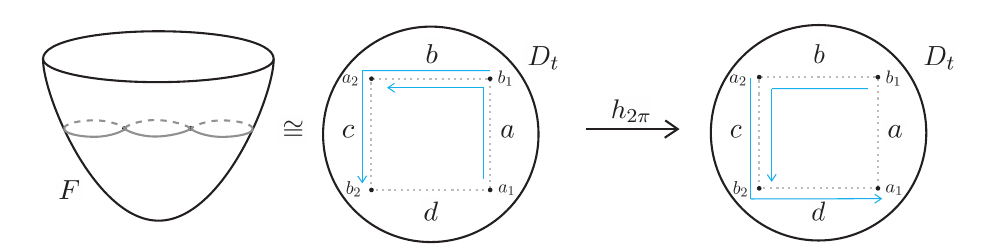}
		\caption{Monodromy of the fiber $F$.}
	\label{pic2}
\end{figure}

We take the cycles $a+b$ and $c+d$ as a basis of $H_1(F;\Z)$. Obviously, $h_*(a+b)=b+c$ and $h_*(b+c)=c+d=-(a+b)$ 
so the matrix of $h_*$ with respect to this basis is:
\[
\left(
\begin{array}{cc}
1  & 0  \\
0  & -1   
\end{array}
\right)
\]
The Lefschetz number is $1$ and hence, any local geometric monodromy must have a fixed point. 
In fact, in our construction there is exactly one fixed point, namely, the origin of the disk $D_t$ which is invariant 
under the $\pi/2$-rotation. As in \cref{triple:ex}, it is not difficult to check that $f\notin\mathfrak{m}_{X,x}^2$.

\end{example}

\section{Applications}

The first application of \cref{nofixedpoints} is a new proof of the following result, originally given by A'Campo. It can be used to show that any hypersurface $(X,x)$ in $\C^{n+1}$ with smooth topological type, must be smooth. 


\begin{corollary}[\textit{cf.} {\cite[Theorem 3]{ACampo1973}}]\label{cor:acampo3}
Let $X\subset \C^{n+1}$ be a germ of a hypersurface, not necessarily smooth at $x\in X$. If the Milnor fiber $F_x$ of $X$ at $x$ has trivial reduced homology with complex coefficients, $\tilde{H}_i(F_x;\C)\cong 0$, then $x$ is a smooth point of $X$.
\end{corollary}
\begin{proof}
Let $f:(\C^{n+1},x)\to(\C,0)$ be the holomorphic germ which gives a reduced equation of $(X,x)$. The Lefschetz number of the local monodromy of $f$ is 1 and, hence, $f\notin \mathfrak{m}_{\C^{n+1},x}^2$, by \cref{nofixedpoints}.
\end{proof}

We recall that two germs of complex spaces $(X,x)$ and $(Y,y)$ in $\C^{n+1}$ have the same topological type if there exists a homeomorphism $\varphi:(\C^{n+1},x)\to(\C^{n+1},y)$ such that $\varphi(X,x)=(Y,y)$.

\begin{corollary}\label{Mumford-type}
Let $(X,x)$ be a germ of hypersurface in $\C^{n+1}$. If $(X,x)$ has the topological type of a smooth hypersurface, then $(X,x)$ is smooth.
\end{corollary}
\begin{proof}
If $(X,x)$ has the topological type of a smooth hypersurface then its Milnor fibre has trivial reduced homology by
\cite[Proposition, p. 261]{Trang1973b}. This implies that $(X,x)$ is smooth by \cref{cor:acampo3}.
\end{proof}

This corollary is related to Zariski's multiplicity conjecture \cite{Zariski1971} which claims that two hypersurfaces in $\C^{n+1}$ with the same topological type have the same multiplicity. Since a hypersurface is smooth if and only if it has multiplicity 1, \cref{Mumford-type} is just a particular case of the conjecture. Zariski showed the conjecture for plane curves but it remains still open in higher dimensions. Another related result is Mumford's theorem \cite{Mumford1961} which states that if $X$ is a normal surface and $X$ is a topological manifold at $x\in X$, then $X$ is smooth at $x$.

Our second application is a no coalescing theorem for families of functions defined on spaces with Milnor property. In \cite{Trang1973}, the second named author showed the following interesting application of  A'Campo's theorem 
(see also \cite{Gabrielov,Lazzeri}). Let $\{H_t\}_{t\in\mathbb C}$ be an analytic family of hypersurfaces defined 
on some open subset $U\subset\mathbb C^n$ with only isolated singularities. Take $B$ a Milnor ball for $H_0$ 
around a singular point $x_0\in H_0$ and assume for all $t$ small enough, the sum of the Milnor numbers of all 
the singular points of $H_t$ in $B$ is constant, that is,
\[
\sum_{x\in H_t\cap B}\mu(H_t,x)=\mu(H_0,x_0).
\]
Then $H_t\cap B$ contains a unique singular point $x$ of $H_t$. The purpose of this section is to prove an adapted 
version of this result in a more general context, namely, for Milnor spaces in the sense of \cite{Hamm-Le-Handbook}:

\begin{definition}
A \emph{Milnor space} is a reduced complex space $X$ such that at each point $x\in X$, the rectified homotopical 
depth $\rhd(X,x)$ is equal to $\dim(X,x)$.
\end{definition}

We refer to \cite{Hamm-Le-Handbook} for the definition of the rectified homotopical 
depth and basic properties of Milnor spaces. In general, $\rhd(X,x)\le \dim(X,x)$, so Milnor spaces are those whose rectified homotopical 
depth is maximal at any point. Some important properties are the following:
\begin{enumerate}
\item any smooth space $X$ is a Milnor space,
\item any Milnor space $X$ is equidimensional,
\item if $X$ is a Milnor space and $Y$ is a hypersurface in $X$ (i.e. $Y$ has codimension one and is defined locally in $X$ by one equation), then $Y$ is also a Milnor space.
\end{enumerate}
As a consequence, any local complete intersection $X$ (not necessarily with isolated singularities) is a Milnor space.
Our setting is motivated by the following theorem due to Hamm 
and Lê (see \cite[Theorem 9.5.4]{Hamm-Le-Handbook}):

\begin{theorem}\label{Hamm-Le} Let $(X,x)$ be a germ of Milnor space and assume that $f\colon(X,x)\to(\mathbb C,0)$ has an isolated critical 
point in the stratified sense. Then the general fibre $F$ of $f$ has the homotopy type of a bouquet of spheres 
of dimension $\dim(X,x)-1$.
\end{theorem}

\begin{corollary}\label{coro: finm2 hamm le} With the hypothesis and notation of \cref{Hamm-Le}, if $f\in\mathfrak{m}^2_{X,x}$, then the trace of the 
induced map $h_*\colon H_{n-1}(F;\mathbb Z)\to H_{n-1}(F;\mathbb Z)$ by the monodromy $h:F\to F$ is $(-1)^n$, where $n=\dim(X,x)$.
\end{corollary}

\begin{definition} With the hypothesis and notation of \cref{Hamm-Le}, the number of spheres of $F$ is 
called the \emph{Milnor number} of $f$ and is denoted by $\mu(f)$. We say that the critical point is \emph{non-trivial} if $\mu(f)>0$.
\end{definition}

We want to generalize the non-coalescing theorem of the second author for families of hypersurfaces $\left\{H_t\right\}_{t\in\C}$ in \cite{Trang1973}. As we want to generalize it in the setting of fibers inside Milnor spaces, we obviously need a convenient concept of \textsl{family of Milnor spaces} and its corresponding \textsl{family of complex functions that give the equations of the fibers}. This is covered in \cref{strat-def}.

Consider a germ of complex analytic space $(X_0,x_0)$. Let $f_0\colon(X_0,x_0)\to(\C,0)$ be a germ of holomorphic function. Let
$X_0$ be a small representative of $(X_0,x_0)$ and let ${\mathcal S}$ be a Whitney stratification of $X_0$. We assume that
a representative $f_0$ has an isolated critical point in the stratified sense at $x_0$. We define:

\begin{definition}\label{strat-def} A \emph{stratified deformation} of $(X_0,x_0)$ is a flat deformation 
$\pi\colon({\mathfrak X}, x_0)\to(\C,0)$, where $\mathfrak X$ is an analytic space with an 
analytic Whitney stratification such that, for a representative $\pi$:
\begin{enumerate} 
\item $\pi^{-1}(0)=X_0$ as analytic spaces,
\item $\pi$ has isolated critical points in the stratified sense,
\item the stratification of $X_0$ coincides with the induced stratification by $\mathfrak X$ on $\pi^{-1}(0)$.
\end{enumerate}

Given a stratified deformation $\pi\colon({\mathfrak X}, x_0)\to(\C,0)$, a \emph{stratified unfolding} of a germ $f_0$ as above is a holomorphic map 
$\mathcal F\colon({\mathfrak X},x_0)\to(\C\times\C,0)$ such that $p_1\circ \mathcal F|_{X_0}=f_0$ and 
$p_2\circ\mathcal F=\pi$, where $p_i:\C\times\C\to\C$, $i=1,2$, is the $i$th-projection.
\end{definition}

We can always assume that $\mathfrak X$ is embedded in $\C^{N}\times\C$ and choose coordinates in such a way 
that $\pi(x,t)=t$. So, we can write the stratified unfolding as $\mathcal F(x,t)=(f_t(x),t)$. For each $t\in\C$, we have 
a function $f_t\colon X_t\to\C$, where $X_t=\pi^{-1}(t)$. Here we consider in $X_t$ the stratification induced by 
$\mathfrak X$ and denote by $\Sigma(f_t)$ the set of stratified critical points of $f_t$. 

\begin{example} We consider the function $f_0:(X_0,0)\to(\C,0)$, where $X_0$ is the surface in $\C^3$ given by $z^2-y(x^2+y)^2=0$ and $f_0(x,y,z)=x$. The stratification in $X_0$ is $\big\{\{0\},C_0-\{0\},X_0\big\}$, where $C_0$ is the curve $z=x^2+y=0$.  It is easy to see $f_0$ has only one critical point in the stratified sense at the origin and that $\mu(f_0)=1$ (see \cref{fig:f0}).

\begin{figure}[ht]
	\centering
		\includegraphics[width=0.5\textwidth]{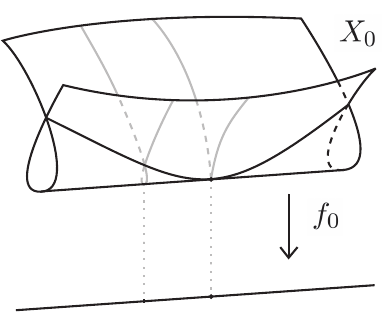}
	\caption{The function $f_0$ with a critical point}
	\label{fig:f0}
\end{figure}

Now we define a stratified deformation $\pi:(\mathfrak X,0)\to(\C,0)$ and a stratified unfolding $\mathcal F\colon({\mathfrak X},0)\to(\C\times\C,0)$ as follows: $\mathfrak X$ is the hypersurface in $\C^3\times\C$ with equation $z^2-y(x^2+y+t)^2=0$, $\pi(x,y,z,t)=t$ and $\mathcal F(x,y,z,t)=(x,t)$. The stratification in $\mathfrak X$ is $\big\{\{0\}, \mathcal D\setminus\{0\},\mathcal C\setminus\mathcal D,\mathfrak X\setminus\mathcal C\big\}$, where $\mathcal D$ is the curve $z=y=x^2+t=0$ and $\mathcal C$ is the surface $z=x^2+y+t=0$. Again it is not difficult to check that all conditions of \cref{strat-def} hold.

For $t\ne0$, $f_t:X_t\to\C$ has two critical points in the stratified sense at $\big(\pm\sqrt{-t},0,0\big)$, which are the points in $D_t:=X_t\cap\mathcal D$. We see that $f_t$ has also Milnor number $1$ at each critical point $\big(\pm\sqrt{-t},0,0\big)$ (see \cref{fig:ft}).

\begin{figure}[ht]
	\centering
		\includegraphics[width=0.57\textwidth]{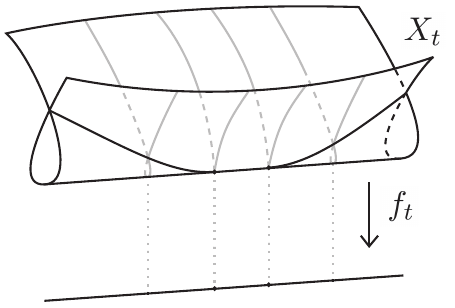}
	\caption{The function $f_t$ with two critical points}
	\label{fig:ft}
\end{figure}

\end{example}

The following theorem could seem very restrictive due to the length of the hypotheses. On the contrary, its statement only says that, with a \textsl{general notion of family of ambient spaces} ($\mathfrak X$) and a \textsl{general notion of equation of the fibers} ($\mathcal F$), if it happens what we have proven in some cases (\cref{nofixedpoints} or \cref{coro: finm2 hamm le}), then we have a non-coalescing result.

\begin{theorem}\label{coales} Let $f_0\colon(X_0,x_0)\to(\C,0)$ be a function with a non-trivial isolated critical point and let 
$\mathcal F\colon(\mathfrak X,x_0)\to(\C\times\C,0)$ be a stratified unfolding of $f_0$ such that $\mathfrak X$ is a Milnor space. We set $Y_t=f_t^{-1}(0)$ 
and assume that for any $x\in\Sigma(f_t)\cap Y_t$, the trace of the local monodromy of $f_t$ at $x$ in dimension $n-1$ is $(-1)^n$, where $\dim (X_0,x_0)=n$.
Let $B_0$ be a Milnor ball for $f_0$ at $x_0$ and assume that for any $t\in\mathbb C$ small enough,
\begin{equation}\label{mu-constant}
\sum_{x\in \Sigma(f_t)\cap Y_t\cap B_0}\mu_x(f_t)=\mu_{x_0}(f_0),
\end{equation}
where $\mu_x(f_t)$ is the Milnor number of $f_t$ at $x$.
Then $Y_t\cap B_0$ contains a unique non-trivial critical point $x$ of $f_t$.
\end{theorem}

\begin{proof} 
Denote by $\Sigma(\mathcal F)$ the set of stratified critical points of $\mathcal F$ and assume that $\dim (X_0,x_0)=n>2$.
It follows that $(x,t)\in\Sigma(\mathcal F)$ if and only if $x\in \Sigma(f_t)$. Since the stratification of $\mathfrak X$ is analytic, $\Sigma(\mathcal F)$ is also analytic. On one hand, we have that
\[
\dim\Sigma(\mathcal F)\cap\{t=0\}=\dim\Sigma(f_0)=0
\] 
and thus, $\dim\Sigma(\mathcal F)\le 1$. On the other hand, by \eqref{mu-constant}, we obtain that
\[
\Sigma(\mathcal F)\cap\{t=t_0\}=\Sigma(f_{t_0})\ne\emptyset,
\] 
for $t_0\ne0$, so $\dim\Sigma(\mathcal F)=1$. Moreover, 
$\mathcal F^{-1}(0)\cap\Sigma(\mathcal F)=\{0\}$, hence its image $\Delta=\mathcal F\big(\Sigma(\mathcal F)\big)$ is also analytic 
of dimension $1$ in $(\C\times\C,0)$ by Remmert's proper map theorem.

We fix a small enough open polydisc $D_\eta\times D_\rho$ in $\C\times\C$ such that the restriction
\begin{equation}\label{fibrationF}
\mathcal F: (B_0\times D_\rho) \setminus \mathcal F^{-1}(\Delta)\to (D_\eta\times D_\rho) \setminus \Delta
\end{equation}
is a proper stratified submersion and such that $\Delta\cap(D_\eta\times\{0\})=\{0\}$. By the Thom-Mather 
first isotopy lemma, \eqref{fibrationF} is a locally $C^0$-trivial fibration. 
Given $s\in D_\eta\setminus\{0\}$, we have $(s,0)\in (D_\eta\times D_\rho) \setminus \Delta$ and hence the fibre,
\[
\mathcal F^{-1}(s,0)\cap (B_0\times D_\rho)=(f_0^{-1}(s)\cap B_0)\times\{0\},
\] 
coincides with the general fibre of $f_0$.

Let $t\in D_\rho$ and assume that $\Sigma(f_t)\cap f_t^{-1}(0)\cap B_0=\{x_1,\dots,x_k\}$. For each $i=1,\dots,k$, 
we take a Milnor ball $B_i$ for $f_t$ at $x_i$ such that $B_i$ is contained in the interior of $B_0$ and 
$B_i\cap B_j=\emptyset$ if $i\ne j$. Now we choose $0<\eta'<\eta$ such that for all $s$, with $0<|s|<\eta'$, $(s,t)\notin \Delta$.

Fix a point $s\in D_{\eta'}$ and consider the loop $\gamma(\theta)=se^{i \theta}$, $\theta\in[0,2\pi]$. 
This loop induces a monodromy $h:f_t^{-1}(s)\cap B_0\to  f_t^{-1}(s)\cap B_0$ which coincides, up to isotopy, 
with the geometric monodromy of $f_0$ at $x_0$. Moreover, by adding the boundaries of the balls $B_i$ as strata in 
the domain of \eqref{fibrationF}, we can assume that:
\begin{enumerate}
\item $h(f_t^{-1}(s)\cap B_i)=f_t^{-1}(s)\cap B_i$ and $h_i=h|_{f_t^{-1}(s)\cap B_i}$ is the monodromy of $f_t$ at $x_i$, for each $1=1,\dots,k$;
\item $h$ is the identity outside the interior of $B_1\cup\dots\cup B_k$.
\end{enumerate}

Let $U=f_t^{-1}(s)\cap\left(B_0 \setminus \bigcup_{i=1}^k\mathring B_i\right)$ and $V=f_t^{-1}(s)\cap\bigcup_{i=1}^kB_i$. 
By considering the Mayer-Vietoris sequence of the pair $(U,V)$ we get a diagram whose rows are exact sequences:
\begin{equation*}
\begin{tikzcd}0\ar[r] &H_{n-1}(U\cap V)\ar[r]\ar[d,"{id}"] &H_{n-1}(U)\oplus H_{n-1}(V)\ar[r]\ar[d,"{id\oplus\big(\oplus_{i=1}^k (h_i)_*\big)}"]  & H_{n-1}(U\cup V)\ar[r]\ar[d,"{h_*}"]&\quad\\
0\ar[r] &H_{n-1}(U\cap V)\ar[r] &H_{n-1}(U)\oplus H_{n-1}(V)\ar[r] & H_{n-1}(U\cup V)\ar[r]&\quad\end{tikzcd}
\end{equation*}
\[
\begin{tikzcd}\ar[r] &H_{n-2}(U\cap V)\ar[r]\ar[d,"{id}"] &H_{n-2}(U)\ar[r]\ar[d,"{id}"]  & 0\\
\ar[r] &H_{n-2}(U\cap V)\ar[r]&H_{n-2}(U)\ar[r]& 0\end{tikzcd}
\]

By the exactness in one of the rows of the sequence we get
\[
a-\left(b+\sum_{i=1}^k \mu_{x_i}(f_t)\right)+\mu_{x_0}(f_0)-c+d=0,
\]
where $a=\rk H_{n-1}(U\cap V)$, $b=\rk H_{n-1}(U)$, $c=\rk H_{n-2}(U\cap V)$ and $d=\rk H_{n-2}(U)$. Our hypothesis implies that
\[
a-b-c+d=0.
\]

Now we use the fact that the trace is additive, which gives:
\[
a-\left(b+\sum_{i=1}^k \tr\big((h_i)_*\big)\right)+\tr(h_*)-c+d=0,
\]
and hence
\[
\sum_{i=1}^k \tr\big((h_i)_*\big)=\tr(h_*).
\]
Again by hypothesis, $\tr\big((h_i)_*\big)=\tr(h_*)=(-1)^n$, for all $i=1,\dots,k$, so necessarily $k=1$.

We can use the same ideas if $n=1$ or $n=2$, with a diagram similar to the one we use above.
\end{proof}

Observe that the hypothesis of having trace equal to $(-1)^n$ at any point can be relaxed to having trace $k\neq 0$ that does not depend on the point. Also, the hypothesis of $\mathfrak{X}$ being a Milnor space is given to assure that the generic fibers of $f_t$ have homology only in middle dimension (by \cref{Hamm-Le}). One can prove something similar if, in general, the non-trivial homology is sparse.

\begin{remark}
The proof of \cref{coales} is an adaptation of the proof given in \cite{Trang1973} for the case $X=\C^n$. A similar argument appears also in the paper \cite{CNOT}, where it is showed that any family of \textsc{icis} with constant total Milnor number has no coalescence of singularities.
\end{remark}

\bibliographystyle{abbrv}
\bibliography{MonodromyBib}

\end{document}